\documentclass[12pt]{article}
\usepackage{mathrsfs,amssymb,amsmath,
amsthm, color, graphicx }
\usepackage{bm}
\usepackage{url} 
\newtheorem{tm}{Theorem}[section]
\newtheorem{lm}[tm]{Lemma}
\newtheorem{co}[tm]{Corollary}
\newtheorem{re}[tm]{Remark}

\newtheorem{exm}[tm]{Example}
\newtheorem{pr}[tm]{Proposition}

\allowdisplaybreaks[4]
\newcommand{\dd}{\rm{d}}
\newcommand{\ee}{\rm{e}}
\makeatletter
\newcommand{\subscripts}[3]{%
  \@mathmeasure\z@\displaystyle{#2}%
  \global\setbox\@ne\vbox to\ht\z@{}\dp\@ne\dp\z@
  \setbox\tw@\box\@ne
  \@mathmeasure4\displaystyle{\copy\tw@_{#1}}%
  \@mathmeasure6\displaystyle{{#2}_{#3}}%
  \dimen@-\wd6 \advance\dimen@\wd4 \advance\dimen@\wd\z@
  \hbox to\dimen@{}\mathop{\kern-\dimen@\box4\box6}%
}
\makeatother

\newcommand{\III}{{\vert \kern-.10em \vert \kern-.10em \vert}}

\def\supp{\mathop{\rm supp}\nolimits}
\makeatletter
 
 \@addtoreset{equation}{section}
\makeatother
 
\begin{document}
\setlength{\baselineskip}
{15.5pt}
%
\allowdisplaybreaks

\title{
A graph discretized approximation of 
semigroups \\
for diffusion with
drift and killing \\
on a complete 
Riemannian manifold 
}
\author{\Large
{Satoshi Ishiwata\hspace{1mm}\footnote{Department of Mathematical Sciences, Faculty of Science, Yamagata University,
1-4-12, Kojirakawa, Yamagata 990-8560, Japan
(e-mail: {\tt ishiwata@sci.kj.yamagata-u.ac.jp})}
~and Hiroshi Kawabi\hspace{1mm}\footnote{Department of Mathematics, Hiyoshi Campus, Keio University,
4-1-1, Hiyoshi, Kohoku-ku, Yokohama 223-8521, Japan (e-mail: {\tt{kawabi@keio.jp}})}
}}

\maketitle
\begin{abstract}
In the present paper, 
we prove that the 
$C_{0}$-semigroup generated by a 
Schr\"odinger operator with drift on a complete Riemannian manifold 
is approximated by the discrete semigroups associated with a family of
discrete time 
random walks with killing in a flow 
on a sequence of proximity graphs, which are
constructed by partitions of the manifold. 
Furthermore, when the manifold is compact, we also obtain a quantitative error 
estimate of the convergence. Finally, we 
give examples of the partition of the manifold and 
the drift term 
on two typical manifolds:
Euclidean spaces and model manifolds.
\vspace{2mm} 
\vspace{2mm} 
\end{abstract}
\section{Introduction}
Let ${M}=({M}, g)$ be a smooth $n$-dimensional 
Riemannian manifold and $\mathfrak m$ be the Riemannian volume measure on $M$.
We assume that $M$ is geodesically complete and connected,
but not necessarily compact.
We denote by $C^{\infty}_{c}(M)$ and $C_{0}(M)$ the spaces of smooth functions on $M$
with compact support
and continuous functions on $M$ vanishing at infinity, respectively.
Let $b$ be a smooth vector field on $M$, and 
$V$ be a non-negative smooth function defined on $M$.
We consider a 
{\it{drifted Schr\"odinger operator}}
${\mathcal A}={\mathcal A}_{V}$ 
having the form
\begin{equation*}
{\mathcal A}f(x)=-\Delta f(x)-(bf)(x)+V(x)f(x),\quad x\in M,~f\in C^{\infty}_{c}(M), 
\end{equation*} 
where $\Delta$ is the negative
Laplacian on $M$.
Since $V$ is non-negative, 
the maximal principle of 
the Laplacian $\Delta$ implies that
$(-{\mathcal A}, C^{\infty}_{c}(M))$ is dissipative, thus is closable in $C_{0}(M)$
(see \cite[Lemma 2.1]{See84} for details).
Moreover 
by the Lumer-Phillips theorem, 
the closure of $(-{\mathcal A}, C^{\infty}_{c}(M))$ generates a contraction $C_{0}$-semigroup
$\{{\ee}^{-t{\mathcal A}} \}_{t \geq 0}$
in $C_{0}(M)$ under the
the condition
\begin{description}
\item[\bf{(A):}]\, 
$(\lambda+{\mathcal A})(C^{\infty}_{c}(M))$ is dense in $C_{0}(M)$ 
for some $\lambda>0$.
\end{description}
See e.g., \cite{Paz85, Ebe99}
for the definition of dissipativity and the Lumer-Philips theorem.

Many problems in analysis on Riemannian manifolds naturally lead to the study of the 
{\it{drifted Schr\"odinger semigroup}} $\{ {\ee}^{-t \mathcal{A}}\}_{t\geq 0}$ and hence 
it is an important problem to find an efficient approximation scheme of ${\ee}^{-t \mathcal{A}} f$ for a given $f \in C_0(M)$.
When $M=\mathbb{R}^n$ and $\mathcal{A}=-\Delta$, 
we have the explicit formula for the heat kernel of the Laplacian $\Delta$. Thus 
the integral representation formula
\begin{equation*}
{\ee}^{t\Delta} f (x)=\int_{\mathbb{R}^n} \frac{1}{(4\pi t)^{n/2}} \exp \big( -\frac{\vert x-y\vert_{\mathbb R^{n}}^2}{4t} 
\big) f(y) {\dd}y, \quad f\in C_{0}(\mathbb R^{n})
\end{equation*}
enables us to study discrete approximations for ${\ee}^{t \Delta} f$ by 
applying 
numerical integration 
methods.
However, such an exact formula of the kernel function of ${\ee}^{-t\mathcal{A}}$ 
for general $M$ and $\mathcal{A}$ 
has been proved only in limited situations.
When $M$ is compact, it is known that the spectrum of the positive Laplacian $-\Delta$ 
consists of eigenvalues 
\begin{equation*}
0=\lambda_0 <\lambda_1 \leq \lambda_2 \leq \cdots
\end{equation*}
and the corresponding (orthonormal) eigenfunctions $\{\phi_i \}_{i=0}^{\infty}$ are smooth on $M$.
Then the heat semigroup $\{ {\ee}^{t \Delta} \}_{t\geq 0}$ can be written by
\begin{equation*}
{\ee}^{t \Delta} f (x)=\sum_{i=0}^\infty e^{-t \lambda_i}
\phi_i (x)
\int_{M} f(y) \phi_{i}(y) {\dd}{\mathfrak m}(y).
\end{equation*}
By this spectral expansion formula, we see that the 
data $\{ \lambda_i\}_{i=0}^\infty$ and $\{ \phi_i\}_{i=0}^\infty$ have an important role
to approximate ${\ee}^{t \Delta}f$.
Needless to say, on compact manifolds, eigenvalues and eigenfunctions themselves 
are interesting objects. 
See, for instance, \cite{Cha84, Ros97, Fuj95, Ots03}.
It is worth mentioning that this kind of argument 
has been received a lot of attention 
in the study of
\textit{manifold learning}.
See \cite{BN03, BN08, BIK14, Tew17, SW17, Ain21} and references therein
for recent results. 
As a matter of fact, these papers motivated us to write the present paper.
In contrast to the above situation, such an expansion for the drifted Schr\"odinger 
semigroup $\{{\ee}^{-t \mathcal{A}}\}_{t\geq 0}$ on a non-compact manifold $M$ 
has been less-developed in general.
\vspace{1mm}

On the other hand, probability theory
provides us a functional integral
point of view to the studies of Schr\"odinger operators on Riemannian manifolds.
We return to the case when the manifold $M$ is 
not necessarily compact and 
consider a diffusion process ${\bf X}=(x_{t}, {\mathbb P}_{x})$ starting from $x\in M$, which is 
generated by $\Delta+b$.
Thanks to 
smoothness of $M$ and $b$, we can construct this diffusion 
process up to the explosion time $\zeta(x):=\inf \{t>0; x_{t}\notin M \}$ 
by solving a stochastic differential equation on $M$. Obviously,
$\zeta(x)=\infty$ for all $x\in M$ when $M$ is compact. 
The diffusion process $\bf X$ 
gives a probabilistic representation for 
the semigroup $\{{\ee}^{-t \mathcal{A}}\}_{t\geq 0}$. More precisely, by 
the Feynman-Kac formula, we have
\begin{equation}
{\ee}^{-t{\cal A}}
f(x)={\mathbb E}^{{\mathbb P}_{x}} \Big[ \exp \Big(\hspace{-0.5mm} -\hspace{-0.5mm}\int_{0}^{t} V(x_{s})\, {\dd}s \Big) f(x_{t}); 
t<\zeta (x) \Big],~~t\geq 0,~x\in M
\label{FK}
\end{equation}
for all $f\in C_{0}(M)$. 
See e.g., \cite{IW89, Gun17} for details. 
In view of Feynman's original path integral approach to quantum physics,
there are many interests of
finite dimensional integral approximations of the 
Feynman-Kac type functional integral (\ref{FK}) over Riemannian manifolds
in many branches of mathematics such as functional analysis, geometric analysis and probability theory.
Actually, this topic has been studied intensively and extensively by many authors. 
See e.g., \cite{ET81, IM85, Ino86, AD99, BP08, MMRS22} and references therein for further related results.
We should mention that, in the case $V=0$, the central limit theorem (CLT, in short) for
geodesic random walks on 
Riemannian manifolds also gives a finite dimensional integral approximation of
 (\ref{FK}). 
For early work in this direction, see e.g., \cite{Jor75, Pin76, Sun81, Blu84}.
\vspace{2mm}

Under these circumstances of finite dimensional integral approximations of the
functional integral (\ref{FK}) together with developments 
of numerical analysis and manifold learning theory, 
it is fundamental and important to 
study a discrete approximation scheme 
for the drifted Schr\"odinger semigroup $\{{\ee}^{-t \mathcal{A}} \}_{t \geq 0}$.
To tackle this problem, 
in the present paper, we introduce a family of
discrete time
{\it{random walks in the flow}} generated by the drift $b$ with {\it{killing}}
on a sequence of proximity graphs, which are
constructed by partitions cutting 
the Riemannian manifold $M$ into small pieces.
Due to the effect of the drift $b$, these random walks are not necessarily
symmetric in general.
This makes our problem difficult and at the same time interesting. 
As a main result, under condition {\bf{(A)}},
we prove that $\{{\ee}^{-t \mathcal{A}} \}_{t \geq 0}$
in $C_{0}(M)$
is approximated by the discrete semigroups
generated by
the family of random walks 
with a suitable scale change (see Theorem \ref{Main-1} for the precise statement).
Furthermore, when $M$ is compact, we also obtain
a quantitative error estimate of the convergence (see Theorem \ref{Main-3}).
As we shall state in Corollary \ref{Main-2}, these results give us a finite dimensional 
{\it{summation}} approximation of the Feynman-Kac type functional integral (\ref{FK})
which would be a theoretical basis of a new numerical method.
\vspace{2mm}

We note that it is possible to study our problem in $L^{p}(\mathfrak m)$-setting
in parallel to $C_{0}(M)$-setting,
where $L^{p}({\mathfrak m})$,
$1<p<\infty$, denotes
the usual real $L^{p}$-space on $M$ with respect to the volume measure ${\mathfrak m}$
equipped with the $L^{p}$-norm 
$\Vert f \Vert_{L^{p}({\mathfrak m})}=\big( \int_{M} \vert f(x) \vert^{p} {\dd} {\mathfrak m}(x) \big)^{1/p}$.
In fact,
under
two conditions
\begin{description}
\item[\bf{(A1)$_{\bm{p}}$:}]~$(-\lambda-{\mathcal A}, C^{\infty}_{c}(M))$ is 
dissipative
in $L^{p}({\mathfrak m})$ for some $\lambda\geq 0$;
\item[\bf{(A2)$_{\bm{p}}$:}]~$(\lambda'+{\mathcal A})(C^{\infty}_{c}(M))$ is dense in $L^{p}({\mathfrak m})$ 
for some $\lambda'>\lambda$,
\end{description}
the closure of $(-{\mathcal A}, C^{\infty}_{c}(M))$ generates a 
$C_{0}$-semigroup $\{{\ee}^{-t{\mathcal A}} \}_{t \geq 0}$ in $L^{p}({\mathfrak m})$,
and we see that most of the arguments in the proof of the above results apply 
$L^{p}(\mathfrak m)$-setting (see also Theorem \ref{Main-1} and 
Corollary \ref{Main-2} for details). 
Note that the semigroup $\{{\ee}^{-t{\mathcal A}} \}_{t \geq 0}$ is not necessarily contractive in
$L^{p}({\mathfrak m})$.
\vspace{1mm} 

We now mention
related works. 
Burago-Ivanov-Kurylev \cite{BIK14} studied a discrete approximation of the Laplacian 
on a compact Riemannian manifold
based on partitions of the manifold.
Although our framework is influenced by this paper, 
we need to extend their argument slightly to the case where 
the underlying manifold is non-compact.
Chen-Kim-Kumagai \cite{CKK13} also studied a similar discrete approximation for a large 
class of symmetric jump processes 
on metric measure spaces satisfying
the volume doubling condition. 
(Note that the volume doubling condition does not necessarily 
hold for general non-compact complete Riemannian manifolds.)
In a series of papers \cite{IKK17, IKN20, IKN21}, we studied CLTs for
non-symmetric random walks on
infinite graphs having a periodic structure such as
crystal lattices and nilpotent covering graphs.
Introducing a family
of random walks which interpolates between the original non-symmetric 
random walk with the symmetrized one, as the limit, we captured the Brownian motion with 
a constant drift on a suitable space in which the graph is realized.
However, this kind of interpolation is different from the family 
of random walks in a flow introduced in the present paper.
It is also worth mentioning
that Kotani \cite{Kot02}
studied a semigroup CLT for 
a generalized Harper operator on a crystal lattice
in $L^{2}$-setting,
and obtained a (uniform) magnetic 
Schr{\"o}dinger semigroup on an Euclidean space as the limit.
\vspace{1mm}

The rest of 
the present paper is organized as follows: In Sections 2.2 and 2.3, we introduce our framework of
the graph discretization of the manifold $M$ and 
a random walk in a flow with killing on the graph. In Section 2.4, 
we state main results (Theorems \ref{Main-1}, \ref{Main-3} and 
Corollaries \ref{Main-2}, \ref{co-weighted}).
In Section 3, we devote ourselves to prove main results. In particular, in 
Theorem \ref{pointwise generator convergence}, we obtain 
a convergence rate of the generators of the 
family of random walks under the suitable scale change 
mentioned above. 
Combining this theorem
with Trotter's approximation theorem (cf. \cite{Tro58, Kur69})
and a recent result on its convergence rate (cf. \cite{Nam22}), we obtain Theorems 
\ref{Main-1} and  \ref{Main-3}, respectively.
In Section 4, we discuss sufficient conditions for conditions {\bf{(A)}}, {\bf{(A1)$_{\bm{p}}$}}
and {\bf{(A2)$_{\bm{p}}$}}.
Finally in Section 5, we give examples of the partition of the manifold and the drift on
two typical manifolds: Euclidean spaces and model manifolds.
\section{Framework and main results}
\subsection{Notations}
We introduce some notations to describe our results in detail.
Let $C(M)$ be the space of all real-valued continuous function on
the Riemannian manifold $M$.
Let 
$C_{0}(M)$ be the subspace of
$C(M)$ vanishing at infinity, i.e., $\lim_{d(x,o) \to \infty}f(x)=0$, where 
$d(x,o)$ denotes
the geodesic distance between $x\in M$ and a base point $o\in M$.
This is a Banach space endowed with the uniform convergence topology
$\Vert f \Vert_{\infty}:=\sup_{x\in M}\vert f(x) \vert$.
Let $C^{\infty}(M)$ be the space of all real-valued smooth functions on $M$.
Obviously, if $M$ is compact, $C_{0}(M)$ and 
$C_{c}^{\infty}(M)$ coincide with $C(M)$ and $C^{\infty}(M)$, respectively.

Let $\partial$ be a {\it{cemetery 
point}} added to 
$M$ so that 
$M_{\partial}:=M \cup \{\partial \}$ is the one-point compactification of $M$.
As usual, we may regard $\partial$ as the point at infinity and  
${\bf{X}}=(x_{t}, {\mathbb P}_{x})$ as a continuous Markov process 
on $M_{\partial}$.
Note that we may rewrite the explosion time as  $\zeta(x)=\inf \{t>0; x_{t}=\partial \}$. 
Since $M$ is complete, $x\to \partial$ is equivalent to
$d(x,o) \to \infty$, 
and thus
$C_{0}(M)$ may be identified with the space of all continuous functions 
$f: M_{\partial} \to \mathbb R$
satisfying $f(\partial)=0$. 
For simplicity of notation, we write $\bm{\partial}=\{ \partial \}$.

Throughout the present paper, we use $c$,
$C$
to denote positive constants which may change from line to line. We also use the Landau symbols 
$O(\cdot)$ and $o(\cdot)$. If the dependence of $C$, $O(\cdot)$ and 
$o(\cdot)$ are significant, we denote them like $C(N)$, $O_{N}(\cdot)$ and $o_{N}(\cdot)$, 
respectively. For $a\in \mathbb R$, we denote by $\lfloor a \rfloor$ (resp. $\lceil a \rceil$)
the greatest integer less than or equal to $a$ (resp. the least integer greater than or equal to $a$).
Unless otherwise specified, we use the Einstein summation convention, 
which means that an index variable that appears twice in an expression is 
implicitly summed over all its possible values.
\subsection{Graph discretization}
In what follows, we present a framework of a graph discretization of the Riemannian manifold $M$. 
For $x \in M$ and $r>0$, let $B_{r}(x)$ be the open geodesic ball of radius
$r>0 $ centered at $x \in M$. 
For any subset $A$ of $ M$, we denote by 
$U_{r}(A)$ 
the $r$-neighborhood of $A $, i.e., 
\begin{equation*}
U_r(A)=\cup_{x \in A} B_r (x).
\end{equation*}
 For two subsets $A$ and $B$ of 
${M}$, we define the Hausdorff distance $d_{H}(A,B)$
between $A$ and $B$ by
$$ d_{H}(A,B):=\inf \{ r>0;~A\subset U_{r}(B),~B \subset U_{r}(A) \}.$$
See \cite[Section 7.3]{BBI01} for details of the Hausdorff distance.
Moreover, we set $$d_{H}(\{\partial \}, A):=\infty, \quad A\subset M.$$

A countable collection $\mathbb{X} $ 
of connected measurable subsets of $M$ with finite measures
is called 
a \textit{partition} of ${M}$
if 
\begin{equation}
M=\bigcup_{X \in {\mathbb X}} X, \quad 
X   \cap Y =\emptyset~~(X,Y \in \mathbb{X}, X \neq Y).
\label{bunkatsu-disjoint}
\end{equation}
Throughout the present paper, we assume
\begin{description}
\item[\bf{(B):}]~$\sharp  \hspace{0.5mm} \big \{ X \in \mathbb{X} ; X \subset B_r (x) \big \}<\infty$
for all $x\in M$ and $r>0$;
\item[\bf{(C):}]~$\vert {\mathbb X} \vert:= \sup_{X \in \mathbb{X}} {\rm diam} (X) <\infty.$
\end{description}For given $x\in M$,
we denote by $X (x)$ the unique element $X\in {\mathbb X}$ containing $x$.
We also define $X(\partial)={\bm{\partial}}$.
For each $X \in \mathbb{X}$, we take a reference point $ x \in X$ and write $x(X)$. 
Set $\mathscr X =\{ x(X) \}_{X \in \mathbb{X}}$. 
It follows from (\ref{bunkatsu-disjoint}) that $x(X) \neq x(Y)$ for $X\neq Y$.

If the manifold $M$ is compact,  condition {\bf{(B)}} implies that such a partition ${\mathbb X}$ is a finite set.
It is always possible to construct a partition 
$\mathbb X$ with {\bf{(B)}} and {\bf{(C)}} in the following manner:
For fixed a constant $\varepsilon >0$, there exists a {\it{maximal 
$\varepsilon$-separated}} subset ${\mathscr X}=\{ x_{i} \}_{i\in \mathbb N}$ 
of $M$. 
Note 
$B_{\varepsilon/2}(x_{i}) \cap 
B_{\varepsilon/2}(x_{j})=\emptyset$, $i\neq j$, and 
$\cup_{i=1}^{\infty} B_{\varepsilon}(x_{i})=M$
(see e.g.,  \cite{Kan85, Kan86} for details).
We set
$$ X_{i}:=B_{\varepsilon} (x_{i}) \setminus \cup_{j<i} B_{\varepsilon}(x_{j}),
\quad i\in {\mathbb N},$$ 
and divide each $X_{i}$ into a finite number of connected components $X_{i}^{1}, \ldots, X_{i}^{k(i)}$.
Obviously, $\mathbb X_{{\mathscr X}}=\{X_{i}^{k}; k=1,\ldots, k(i), i\in \mathbb N \}$ is a partition of $M$ 
satisfying {\bf{(C)}} because  $\vert {\mathbb X}_{\mathscr X} \vert \leq 2\varepsilon$. Besides,
since ${\mathscr X}$ is $\varepsilon$-separated and $M$ is complete and smooth, 
we easily see that $B_{r}(x) \cap {\mathscr X}$ is a finite set 
for all $x\in M$ and $r>0$. This means that ${\mathbb X}_{\mathscr X}$ satisfies {\bf{(B)}}.
As another example of the partition of $M$, we may 
consider the {\it{Voronoi decomposition}} of $M$ with respect to ${\mathscr X}$, which is
defined by $$ X_{1}:={\widetilde{X_{1}}},\quad X_{i}:={\widetilde{X_{i}}}\setminus \cup_{j=1}^{i-1} X_{j},~~i=2,3,\ldots,
$$
where
$${\widetilde{X_{i}}}:=\big \{ y\in M;~d(x_{i},y) \leq d(x_{j}, y), ~j\in {\mathbb N}, i\neq j \big\}, \quad
i\in {\mathbb N}.$$

For given $\rho>0$, let $\mathbb{X}$ be a partition of $M$ with ${\bf{(B)}}$ and ${\bf{(C)}}$
satisfying $\vert \mathbb X \vert < \rho$. 
We say that $X$ and $Y$ in $\mathbb{X}$ are \textit{adjacent} if $d_H(X,Y)<\rho$ and write $X\sim_\rho Y$. 
Then we define an oriented graph $\mathbb{G}(\mathbb{X}, \rho)=(\mathbb{V}, \mathbb{E})$ called 
a ($\rho$-) {\it{proximity graph}} of the manifold $M$ by $\mathbb{V}: ={\mathbb X}$ and 
$$\mathbb{E}:=\{ (X, Y)\in \mathbb{X}\times \mathbb{X}; X \sim_\rho Y \}.$$ 
Since $M$ is connected, ${\mathbb G}({\mathbb X}, \rho)$ is also connected, and
{\bf{(B)}} implies that ${\mathbb G}({\mathbb X}, \rho)$ is
locally finite, that is, for any $X\in \mathbb{X}$, 
 the ($\rho$-)neighborhood of  $X$ defined by 
$$N_{\rho}(X):=
\{ 
Z \in \mathbb{X} ; Z \sim_\rho X 
\}
$$
is a finite set. We should mention that 
${\mathbb G}({\mathbb X}, \rho)$ is uniform, i.e.,  
$\sup_{X\in {\mathbb X}} \sharp N_{\rho}(X)<\infty$, 
provided that the Ricci curvature of $M$ is bounded from below.
See \cite[Lemma 2.3]{Kan85} for details.

For later purpose, we introduce a 
graph ${\mathbb G}_{\partial}({\mathbb X}, \rho)
=({\mathbb V}_{\partial}, {\mathbb E}_{\partial})$ 
constructed by the sum of 
the proximity graph $\mathbb G({\mathbb X}, \rho)$ and the point $ \partial $ mentioned above. 
To be precise,
${\mathbb V}_{\partial}:={\mathbb V} \cup 
{\bm{\partial}}$
and
$${\mathbb E}_{\partial}:={\mathbb E} \cup \{ (X, 
{\bm{\partial}}),  ({\bm{\partial}}, X);  X \in \mathbb{X} \}
\cup \{ ({\bm{\partial}}, {\bm{\partial}}) \}.$$ 
We denote by $C_{0}({\mathbb G}_{\partial} ({\mathbb X}, \rho))$ the space of all
bounded functions $F: {\mathbb X} \cup \{ {\bm{\partial}} \}
\to \mathbb R$ with $F({\bm{\partial}})=0$ endowed with  
$$\Vert F \Vert_{C_{0}({\mathbb G}_{\partial}({\mathbb X}, \rho))}:=\sup_{X \in {\mathbb X}} \vert F(X) \vert,$$ 
and by $L^{p}({\mathbb G}_{\partial} ({\mathbb X}, \rho))$, $1<p<\infty $, the set of all functions 
$F: {\mathbb X}  \cup 
\{ {\bm{\partial}} \}
\to \mathbb R$ satisfying $F({\bm{\partial}})=0$ and 
$$ \Vert F \Vert_{L^{p}({\mathbb G}_{\partial} ({\mathbb X}, \rho))}:=
\Big( \sum_{X \in {\mathbb X}} \vert F(X) \vert^{p} {\mathfrak m}(X) \Big)^{1/p}<\infty.$$ 

For a given partition $\mathbb X$
and a set of reference points ${\mathscr X} =\{ x(X) \}_{X \in \mathbb{X}}$ associated with 
$\mathbb X$, 
we define two kinds of {\it{discretization maps}}
${\mathcal P}_{\mathbb X}: L^p({\mathfrak m}) 
\rightarrow L^p({\mathbb G}_{\partial} (\mathbb{X}, \rho))$ 
and 
$[\cdot ]_{\mathscr X} : C_{0}(M) \rightarrow C_{0}({\mathbb G}_{\partial} (\mathbb{X}, \rho))$
by
\begin{equation*}
{\mathcal P}_{\mathbb X}f(X)=\frac{1}{{\mathfrak m}(X)}\int_{X} f 
\, {\dd}{\mathfrak m}
\quad \mbox{and}
\quad
[f]_{{\mathscr X}}
(X)=f(x(X)),\quad 
\end{equation*}
respectively.
\subsection{Random walk in a flow with killing}
Let us recall that $b$ is a smooth vector field on the 
complete Riemannian manifold $M$ and $V$ is a non-negative smooth
potential function on $M$. 

We first introduce the notion of the flow generated by $b$.
For each $x\in M$, there exist 
an open interval $I=I_{x}=(-t(x), t(x))$
around $0\in \mathbb R$ and a 
smooth curve $(\varphi_{t}(x))_{t\in I}$ on $M$ satisfying
\begin{equation*}
\varphi_{0}(x)=x, \quad
\frac{d}{dt} \varphi_{t}(x)=b(\varphi_{t}(x)), \quad t\in I.
\end{equation*}
Since the solution to the above differential equation depends smoothly on the initial point $x\in M$,
it induces a local flow $\varphi=(\varphi_{t}(x)): \{(t,x);~t\in I_{x}, x\in M \}\to M$ generated by the vector
field $b$. Note that $t(x)$
is continuous with respect to $x$.
We set $$s(x):=\inf \{ t>0; \varphi_{t}(x)\notin M\}, \quad x\in M.$$
Obviously, 
$s(x)\geq t(x)$
holds for all $x\in M$. 
Moreover $\varphi_{s}(x)=\partial$ provided 
$s\geq s(x)$.

For $\vert {\mathbb X} \vert <\rho$, $X \in \mathbb{X}$ and  $s\geq 0$, 
we set
\begin{equation*}
N_{\rho, {\mathscr X}}(s; X):=\{ Y \in  \mathbb{X} ; Y \sim_\rho X(\varphi_s(x(X))) \}
\end{equation*}
and 
\begin{equation*}
\mathcal{N}_{\rho, {\mathscr X}}(s; X):=\bigcup_{Y \in N_{\rho, {\mathscr X}}(s; X)} Y.
\end{equation*}
In the case $b=0$, $N_{\rho, {\mathscr X}}(s; X)=N_{\rho}(X)$. If $s \geq s(x(X))$, 
 $N_{\rho, {\mathscr X}}(s; X)={\bm{\partial}}$.

Inspired by an idea of \cite{Mad89}, 
we introduce a random walk in the flow $\varphi$ with killing on 
the proximity graph
$({\mathbb G}_{\partial}({\mathbb X}, \rho), {\mathscr X})$.
Let $\alpha, s\geq 0$. We
define the transition probability $p_{\alpha, s}=p_{\alpha, s, \mathscr X}$
on $({\mathbb G}_{\partial}({\mathbb X}, \rho), {\mathscr X})$ as follows:
For each $X \in \mathbb{V}$, $Y \in \mathbb{V}_{ \partial }$ and
$0\leq s<  s(x(X))$,
\begin{equation*}
p_{\alpha, s}(X,Y):=\left\{ 
\begin{array}{ll}
{\displaystyle{
\min \{\alpha V(x(X)), 1\}
}} & \mbox{if } Y= {\bm{\partial}}, 
\\
{\displaystyle{
\max \big \{
1-\alpha V(x(X)), 0 \big \}
\frac{\mathfrak m(Y)}{\mathfrak m (\mathcal{N}_{\rho, {\mathscr X}}(s; X))}
}}
 &  \mbox{if }Y \in N_{\rho, {\mathscr X}}(s; X), \\
0 & \mbox{otherwise}
\end{array}\right.
\end{equation*}
and for $s \geq s(x(X))$
\begin{equation*}
p_{\alpha, s}(X, Y):=\left\{ 
\begin{array}{ll}
1 &  \mbox{if }Y={\bm{\partial}}, \\
0 & \mbox{otherwise}
\end{array}\right. .
\end{equation*}
Besides, we set
$$ p_{\alpha, s}(
{\bm{\partial}},
{\bm{\partial}}):=1, \quad p_{\alpha, s}({\bm{\partial}}, X):=0, \quad X 
\in {\mathbb X} ~~~\mbox{for all }s\geq 0.$$
In the usual manner, the transition probability $p_{\alpha, s}$ induces a 
time homogeneous Markov chain on ${\mathbb G}_{\partial}({\mathbb X}, \rho)$. We call it the {\it{random walk in the flow}} 
$\varphi_{s}$ with a {\it{killing rate}} $\alpha V$.

The corresponding transition operator $L_{\alpha, s}
=L_{\alpha, s, \mathscr X}: C_{0}(
{\mathbb G}_{\partial}({\mathbb X}, \rho)) \to 
C_{0}({\mathbb G}_{\partial} ({\mathbb X}, \rho))$ 
is given by 
$L_{\alpha, s}F({\bm{\partial}}):=F({\bm{\partial}})=0$ and
\begin{align}
L_{\alpha, s} F(X)&:=
\sum_{Y\in {\mathbb V}_{\partial}} 
p_{\alpha, s}(X, Y)F(Y), \quad X\in {\mathbb X}.
\label{transition}
\end{align}
This means 
\begin{equation*}
L_{\alpha, s} F(X)
=\left\{ 
\begin{array}{ll}
{\displaystyle{
\max \big \{
1-\alpha V(x(X)), 0 \big \}
\sum_{Y \in N_{\rho, {\mathscr X}}(s; X)}
\frac{{\mathfrak m}(Y)}{{\mathfrak m}(\mathcal{N}_{\rho, \mathscr{X}}(s; X ))} F(Y)
}}
&  \mbox{if }s<s(x(X)), \\
0 & \mbox{if }s\geq s(x(X)).
\end{array}\right. 
\end{equation*}
We abbreviate $L_{\alpha, \alpha}$ to $L_{\alpha}$ for brevity.

Here we give a remark. Let $b=0$ and $V=0$, and
put
$m(X):={\mathfrak m}(X)
{\mathfrak m}(\mathcal{N}_{\rho, \mathscr{X}}(s; X))$, $X \in \mathbb{X}$,
and ${m}({\bm{\partial}})={\mathfrak m}({\bm{\partial}}):=0$.
Then we easily see
$$ p_{\alpha, s}(X,Y)m(X)={\mathfrak m}(X){\mathfrak m}(Y)=p_{\alpha, s}(Y, X)m(Y), \quad X, Y\in {\mathbb V}_{\partial},$$
which means that the corresponding random walk is {\it{$m$-symmetric}}.
On the other hand, 
the random walk in the flow $\varphi$ is not
necessarily symmetric in general.
Indeed, by taking $s>0$ large enough and $\rho>0$ small enough, the distances
among $x \in X$, $\varphi_{s}(x)$ and $\varphi_{2s}(x)$ 
can be very large. In this situation, if $Y$ is close to $\varphi_{s}(x)$, we observe 
$p_{\alpha, s}(X, Y)>0$ and $p_{\alpha, s}(Y, X)=0$. This
means that the random walk cannot be symmetric.
\subsection{Main results}
Now we are in a position to state the main 
results 
in the present paper.
We start with presenting the following theorem.

\begin{tm}
\label{Main-1}
Let $M$ be an $n$-dimensional complete manifold, 
$b$ be a smooth vector field on $M$ and $V$ be a 
non-negative smooth potential function on $M$.
Let $\mathcal{A}=-\Delta-b+V$ be a drifted Schr\"odinger operator on $M$ 
satisfying condition {\bf{(A)}}.
Let $\{\mathbb{X}_{k} \}_{k\in {\mathbb N}}$ 
be a sequence of partitions of $M$ satisfying 
conditions {\bf{(B)}}, {\bf{(C)}} and 
$ |\mathbb{X}_{k}|  \searrow 0$ as $k\to \infty$.
Let $\{ k(\rho) \}_{\rho>0}$
be a subsequence of $\mathbb N$
satisfying
$k(\rho)\nearrow \infty$ 
and $|\mathbb{X}_{k(\rho)}|=o(\rho^2)$ as 
$\rho \searrow 0$. 
Then 
for any 
$f\in C_{0}(M)$ and $t>0$,
\begin{equation}
\lim_{\rho \searrow 0}
\Big \Vert
L_{\frac{\rho^2}{2(n+2)}}
^{\lfloor \frac{2(n+2)}{\rho^2}t \rfloor} 
[f]_{{\mathscr X}_{k(\rho)}}
-\big[ {\ee}^{-t\mathcal A }  f \big]_{{\mathscr X}_{k(\rho)}}
\Big \Vert_{C_0({\mathbb G}_{\partial} (\mathbb{X}_{k(\rho)}, \rho))}
=0.
\label{Main conv1}
\end{equation}
Moreover,
this convergence also holds in $L^{p}$-setting for all $1<p<\infty$.
Namely, under conditions {\bf{(A1)$_{\bm{p}}$}}, {\bf{(A2)$_{\bm{p}}$}}, {\bf{(B)}} and  {\bf{(C)}}, for any 
$f \in L^{p}(\mathfrak m)$ and $t>0$,
\begin{equation}
\lim_{\rho \searrow 0}
\Big \Vert
L_{\frac{\rho^2}{2(n+2)}
}^{\lfloor \frac{2(n+2)}{\rho^2}t \rfloor} {\mathcal P}_{{\mathbb X}_{k(\rho)}} f
-
{\mathcal P}_{{\mathbb X}_{k(\rho)}} ({\ee}^{-t\mathcal A}  f) 
\Big \Vert_{L^p({\mathbb G}_{\partial} (\mathbb{X}_{k(\rho)}, \rho))} 
=0.
\label{Main conv2}
\end{equation}
\end{tm}

This theorem implies the following corollary immediately.
\begin{co} \label{Main-2}
Under the setting of Theorem {\rm{\ref{Main-1}}}, 
we have the following:
\\
{\rm{(1)}}~For any $x\in M$, $f\in C_{0}(M)$ and $t>0$,
\begin{equation}
{\ee}^{-t{\mathcal A}}f(x)=\lim_{\rho \searrow 0}
L_{\frac{\rho^2}{2(n+2)}}^{\lfloor \frac{2(n+2)}{\rho^2}t \rfloor}
[f]_{{\mathscr X}_{k(\rho)}}(X_{k(\rho)}(x)),
\label{pointwise conv}
\end{equation}
where $X_{k(\rho)}(x)$ is the unique element of $\mathbb{X}_{k(\rho)}$ containing $x$. 
\\
{\rm{(2)}}~For any bounded open set $U$ in $M$, $f\in L^{p}({\mathfrak m})$ and $t>0$,
\begin{equation}
\int_{U} 
{\ee}^{-t{\mathcal A}}f(x) {\dd} {\mathfrak m}(x)=
\lim_{\rho \searrow 0}
\sum_{X\in {\mathbb X}_{k(\rho)},  X\subset U}
L_{\frac{\rho^2}{2(n+2)}}^{\lfloor \frac{2(n+2)}{\rho^2}t \rfloor}
{\mathcal P}_{{\mathbb X}_{k(\rho)}} f(X) 
{\mathfrak m}(X).
\label{mean conv}
\end{equation}
\end{co}
We also obtain 
the following corollary by repeating arguments in the proof of Theorem \ref{Main-1} with a slight modification.
\begin{co}
\label{co-weighted}
{\rm{(1)}}~
Suppose that the potential function $V$ is bounded from below by $-v_0$ for some constant $v_0 \geq 0$.
Then
under the same setting as Theorem {\rm{\ref{Main-1}}}, 
\begin{equation*}
\lim_{\rho \searrow 0}
\Big \Vert
e^{v_0 t}
\widetilde{L}_{\frac{\rho^2}{2(n+2)}}
^{\lfloor \frac{2(n+2)}{\rho^2}t \rfloor} 
[f]_{{\mathscr X}_{k(\rho)}}
-\big[ {\ee}^{-t\mathcal{A} }  f \big]_{{\mathscr X}_{k(\rho)}}
\Big \Vert_{C_0({\mathbb G}_{\partial} (\mathbb{X}_{k(\rho)}, \rho))}
=0,
\end{equation*}
where $\widetilde{L}_\alpha=\widetilde{L}_{\alpha, \alpha}$ is the transition operator in the flow $\varphi_{\alpha}$
with the killing rate  $\alpha(V+v_0)$
defined in {\rm{(\ref{transition})}}. 
\vspace{2mm} \\
{\rm{(2)}}~
Let $(M, g, \mu)$ be a weighted manifold with ${\dd}\mu (x)=e^{-U(x)} {\dd} \mathfrak{m}(x)$ for a smooth function $U$ on $M$.
Define the weighted Laplacian $\Delta_\mu$ by
\begin{equation*}
\Delta_{\mu}f= \mathrm{div}_\mu( \nabla f)=\Delta f - \langle \nabla U, \nabla f \rangle,
\end{equation*}
where $\mathrm{div}_\mu  b= e^{U} \mathrm{div}\left( e^{-U} b \right)$ (see {\rm{\cite[Section 3.6]{Gri09}}} and 
{\rm{\cite{Shi12}}}).
Then Theorem {\rm{\ref{Main-1}}} with respect to the weighted measure $\mu$ instead of $\mathfrak{m}$
(used in the definition of $L_{\alpha, s}$) and $\mathcal{A}_\mu =-\Delta_\mu -b + V$ holds.
In this situation, positive constants $K_1$ and $K_2$ appearing in the error estimate 
{\rm{(\ref{pointwise generator estimate})}} in Theorem {\rm{\ref{pointwise generator convergence}}}
depend also on $U$.
\end{co}

Now we further impose that the manifold $M$ is compact. In this case, we have the following
quantitative error estimate.
\begin{tm}
\label{Main-3}
Let $M$ be an $n$-dimensional closed Riemannian manifold.
Under the setting of Theorem {\rm{\ref{Main-1}}}, we take a subsequence $\{ k(\rho) \}$ of $\mathbb{N}$ satisfying 
$k(\rho) \nearrow \infty$ and 
$| \mathbb{X}_{k(\rho)} |=O(\rho^{2+\alpha})$ as $\rho \searrow 0$, where $\alpha>0$. Then 
for any $f\in C^{\infty}(M)$, $t\geq 0$ and 
$0<\rho<1$ satisfying condition
{\rm{(\ref{condition of V})}} in Theorem {\rm{\ref{pointwise generator convergence}}} below, 
the following error estimate holds:
\begin{equation}
\Big \Vert
L_{\frac{\rho^2}{2(n+2)}}
^{\lfloor \frac{2(n+2)}{\rho^2}t \rfloor} 
[f]_{{\mathscr X}_{k(\rho)}}
-\big[ {\ee}^{-t\mathcal A }  f \big]_{{\mathscr X}_{k(\rho)}}
\Big \Vert_{C_0({\mathbb G}_{\partial} (\mathbb{X}_{k(\rho)}, \rho))}
\leq C \rho^{\alpha \wedge 1}, 
\label{Main conv3}
\end{equation}
where $C=C(t, n, f, 
\Vert b \Vert_{\infty}, \Vert \nabla_{b}b \Vert_{\infty}, \Vert V \Vert_{\infty})
$ is a positive constant.
\end{tm}
\section{Proof of main results}
\subsection{Convergence of generators}
\begin{lm} 
\label{approximation-jyunbi}
Let $\{ \mathbb{X}_k \}_{k=1}^{\infty}$ be a sequence of partitions of $M$ satisfying 
$ |\mathbb{X}_k|  \searrow 0$ 
as $k \to \infty$, and ${\mathscr X}_{k}$ be a set of reference points associated with $\mathbb X_{k}$.
Then the sequence of spaces $L^p({\mathbb G}_{\partial} (\mathbb{X}_k,\rho))$ and 
$C_0({\mathbb G}_{\partial }(\mathbb{X}_k, \rho))$ 
together with the maps ${\mathcal P}_{{\mathbb X}_{k}}: L^p(\mathfrak m) \rightarrow 
L^p({\mathbb G}_{\partial} (\mathbb{X}_k, \rho))$ and 
$[\cdot ]_{{\mathscr X}_{k}}
: C_{0}({M}) \rightarrow C_{0}({\mathbb G}_{\partial} (\mathbb{X}_k, \rho))$ 
 approximate $L^p(\mathfrak m)$  and $C_{0}({M})$ in the sense of 
Trotter {\rm{\cite{Tro58}}}, respectively. 
Namely, for any $f \in L^p(\mathfrak m)$
\begin{align}
\| 
{\mathcal P}_{{\mathbb X}_{k}} f \|_{L^p({\mathbb G}_{\partial} (\mathbb{X}_k, \rho))} 
\leq & \| f \|_{L^p(\mathfrak m)}, \label{Lp1}\\
\lim_{k \rightarrow \infty}\| 
{\mathcal P}_{{\mathbb X}_{k}} f \|_{L^p({\mathbb G}_{\partial} (\mathbb{X}_k, \rho))} 
=& \| f \|_{L^p(\mathfrak m)},  \label{Lp2}
\end{align}
and for any $f \in C_{0}(M)$
\begin{align}
\| [f]_{{\mathscr X}_{k}} \|_{C_0({\mathbb G}_{\partial} (\mathbb{X}_k, \rho))} \leq & \| f \|_{C_0(M)},   \label{C_01}\\
\lim_{k \rightarrow \infty}\|  [f]_{{\mathscr X}_{k}} 
\|_{C_0({\mathbb G}_{\partial} (\mathbb{X}_k, \rho))} =& \| f \|_{C_0(M)}. \label{C_02}
\end{align}
\end{lm}
\begin{proof}
The estimate (\ref{Lp1}) follows directly from H\"older's inequality. Noting that
$\{{\mathscr X}_{k} \}_{k=1}^{\infty}$
satisfies 
$d_{H}({\mathscr X}_{k}, M) \rightarrow 0$ as $k \to \infty$, we also obtain
the estimate (\ref{C_01}) and the convergence (\ref{C_02}) immediately.

To prove (\ref{Lp2}), it suffices to show that for any $\varepsilon>0$ and $f \in L^p(\mathfrak m)$ 
there exists $k_0 \in \mathbb{N}$ such that 
for any $k\geq k_0$
\begin{equation}
\| f \|_{L^p(\mathfrak m)} - \| {\mathcal P}_{{\mathbb X}_{k}} f \|_{L^{p}({\mathbb G}_{\partial} (\mathbb{X}_k, \rho))} < \varepsilon.
\label{upper}
\end{equation}
Moreover, since $C_{c}^\infty (M)$ is dense in $L^p(\mathfrak m)$, 
it suffices to show the estimate in (\ref{upper}) for any $f_0 \in C_{c}^\infty(M)$. 
Indeed, for any $f \in L^p(\mathfrak m)$ 
take $f_0 \in C_{c}^\infty(M)$ such that $\| f -f_0 \|_{L^p(\mathfrak{m})} <\varepsilon$.  Then we obtain
\begin{eqnarray}
\|f \|_{L^p(\mathfrak m)} -\| {\mathcal P}_{{\mathbb X}_{k}} f 
\|_{L^p({\mathbb G}_{\partial} (\mathbb{X}_k, \rho )) }
&=&
\| f-f_0 +f_0 \|_{L^p(\mathfrak m)} -\| {\mathcal P}_{{\mathbb X}_{k}}
( f-f_0 +f_0 ) \|_{L^p({\mathbb G}_{\partial} (\mathbb{X}_k, \rho))} 
\nonumber \\
&\leq & \| f- f_0 \|_{L^p(\mathfrak m)} +\| f_0 \|_{L^p(\mathfrak m)}
\nonumber \\
&\mbox{ }&~
 - \| {\mathcal P}_{{\mathbb X}_{k}} 
  f_0 \|_{L^p({\mathbb G}_{\partial} (\mathbb{X}_k, \rho))} 
+\| {\mathcal P}_{{\mathbb X}_{k}} 
(f-f_0) \|_{L^p({\mathbb G}_{\partial} (\mathbb{X}_k, \rho))} 
\nonumber \\
&\leq & 
2 \| f-f_0 \|_{L^p(\mathfrak m)} + \| f_0 \|_{L^p(\mathfrak m)} - \|
{\mathcal P}_{{\mathbb X}_{k}} 
 f_0 \|_{L^p({\mathbb G}_{\partial} (\mathbb{X}_k , \rho))} 
\nonumber \\
&\leq & 2\varepsilon + \| f_0 \|_{L^p(\mathfrak m)} - \|
{\mathcal P}_{{\mathbb X}_{k}} 
f_0 \|_{L^p({\mathbb G}_{\partial} (\mathbb{X}_k , \rho))}.
\nonumber
\end{eqnarray}
To prove 
$\| f_0 \|_{L^p(\mathfrak m)} - \| {\mathcal P}_{{\mathbb X}_{k}} 
 f_0 \|_{L^p({\mathbb G}_{\partial} (\mathbb{X}_k, \rho))} \searrow 0$, 
it suffices to prove that 
$$
 \| f_0 \|_{L^p(\mathfrak m)}^p - \| {\mathcal P}_{{\mathbb X}_{k}}  
  f_0 \|^p _{L^p({\mathbb G}_{\partial} (\mathbb{X}_k, \rho ))} \searrow 0$$ 
by the continuity of function $x^{1/p}$ defined on $[0, \infty)$.
Then we have 
\begin{align*}
& \| f_0 \|_{L^p(\mathfrak m)}^p - \|
{\mathcal P}_{{\mathbb X}_{k}} 
f_0 \|^p _{L^p({\mathbb G}_{\partial} (\mathbb{X}_k, \rho ))} 
 \nonumber  \\
 &=
\sum_{X \in  \mathbb{X}_k} 
\left( 
\frac{1}{\mathfrak{m}(X)} \int_{X} |f_0|^p \, {\dd}{\mathfrak m} 
- \left| \frac{1}{\mathfrak{m}(X)} \int_{X} f_0 \, {\dd}{\mathfrak m} \right|^p  
\right) \mathfrak{m}(X)\\
& \leq 
{\mathfrak m} \big(B_{2\rho}(\supp f_0) \big) 
\max_{X \subset B_{2\rho}(\supp f_0)}
\left( 
\frac{1}{{\mathfrak m}(X)} \int_{X} |f_0|^p \, {\dd}{\mathfrak m} - 
\left| \frac{1}{\mathfrak{m}(X)} \int_{X} f_0 \, {\dd}{\mathfrak m}
 \right|^p  
\right) .
\end{align*}
Since $X \in \mathbb{ X }_k$ is connected, the integral-type mean value theorem implies that 
there exists $z \in X$ such that 
\begin{equation*}
\frac{1}{\mathfrak{m}(X)} \int_{X} f_{0} \, {\dd}{\mathfrak m}=f_{0}(z).
\end{equation*}
Then we obtain
\begin{equation*}
\frac{1}{\mathfrak{m}(X)} \int_{X} |f_0|^p \, {\dd}{\mathfrak m} - 
\left| \frac{1}{\mathfrak{m}(X)} \int_{X} f_0 \, {\dd}{\mathfrak m}
\right|^p  
=\frac{1}{\mathfrak{m}(X)} \int_{X} 
\big( |f_0(y)|^p - |f_0(z ) |^p \big) \, {\dd} {\mathfrak m}(y).
\end{equation*}
Since $f_0$ is uniformly continuous, for any $\varepsilon>0$, there exists 
$\delta>0$ such that if $ \mathrm{diam} X< \delta $
then 
\begin{equation*}
| f_0(y) |^p \leq (|f_0(z)|+\varepsilon)^p.
\end{equation*}
Then we obtain for $y \in X$
\begin{equation*}
|f_0(y)|^p -|f_0(z)|^p \leq p(|f_0(z)| +\varepsilon)^{p-1} \varepsilon,
\end{equation*}
which concludes that 
\begin{equation*}
\max_{X \subset B_{2\rho }(\supp f_0)}
\left( 
\frac{1}{\mathfrak{m}(X)} \int_{X} |f_0|^p \, {\dd}{\mathfrak m} - \left| \frac{1}{\mathfrak{m}(X)} \int_{X} f_0 
\, {\dd}{\mathfrak m}
 \right|^p  
\right) 
\leq \varepsilon p\big (  \sup_{x\in M} |f_0(x)| +\varepsilon \big)^{p-1} .
\end{equation*}
Hence, the proof of the convergence in (\ref{Lp2}) is completed. 
\end{proof}
Let $f\in C^{\infty}_{c}(M)$ and $b$ be a smooth vector field on $M$. We denote by $\varphi$ the flow 
generated by $b$.
Since $\overline{U_2(\supp f)}$ is compact and $t=t(x)$ is a positive continuous function,
$$s(f):=\min \big \{ t(x) \vert \hspace{0.5mm} x\in\overline{U_2(\supp f)} \big \}$$ 
is positive, and 
$\varphi_{s}$ gives a diffeomorphism between 
$\varphi_{-s}U_2(\supp f)$ and $U_2(\supp f)$
for any $0\leq s \leq s(f)$.
We then define a \textit{generalized support of $f$ in the flow $\varphi$} by
$${S}(f):= \overline{\bigcup_{0\leq s \leq \min\{ s(f),  1\} } \varphi_{-s }U_{2} 
\big( \supp f \big)},
$$
and set $$\Vert b \Vert_{\infty, S(f)}=\max_{x\in S(f)} \vert b(x) \vert_{T_{x}M}.$$
Moreover, for a non-negative smooth potential function $V$ on $M$, we also set 
$$\Vert V \Vert_{\infty, S(f)}=\max_{x\in S(f)} V(x).$$
We emphasize that $S(f)$ is compact, $0\leq 
\Vert b \Vert_{\infty, S(f)}<\infty$ and
$0\leq \Vert V \Vert_{\infty, S(f)}<\infty$.

For $f\in C^{\infty}_{c}(M)$, 
we set 
\begin{equation*}
\| f \|_{C^k} := \max_{ |\alpha| \leq k} \sup_{x \in M} 
\left| \frac{\partial^{|\alpha|} f}{\partial^{\alpha_1} u^{(1)}  \cdots \partial^{\alpha_n} u^{(n)}}(x) \right|,
\quad k\in \mathbb N,
\end{equation*}
where $\alpha=(\alpha_1,  \cdots ,  \alpha_n)$ is the multi-index, $|\alpha |=\alpha_1+ \cdots +\alpha_n$ and 
$(u^{(1)}, \cdots ,u^{(n)})$ is the Riemannian normal coordinates at $x \in M$.

The following theorem plays a crucial role in the proof of our main results.
\begin{tm} \label{pointwise generator convergence}
Let $f\in C^{\infty}_{c}(M)$, and 
take $\rho>0$ small enough such that  
\begin{equation}
\frac{ \rho^2}{2(n+2)} \leq \min \left\{ s(f), \| V \|_{\infty, S(f)}^{-1}, 1 \right\}
\label{condition of V}
\end{equation}
and
\begin{equation*}
\rho<\mathrm{inj}_M(S(f)),
\end{equation*}
where $\mathrm{inj}_M(S(f))$ is the infimum of the injectivity radius in $S(f)$.
Then for any partition ${\mathbb X}$ satisfying
$| \mathbb{X} | <\rho/3$, there exist 
positive constants 
${K}_{1}$ and ${K}_{2}$
depending 
on $\|f \|_{C^3}$, $\Vert b \Vert_{\infty, S(f)}$, $\Vert \nabla_{b}b \Vert_{\infty, S(f)}$,
$\Vert V \Vert_{\infty, S(f)}$ and 
$S(f)$ 
such that 
\begin{align}
\left| 
\frac{2(n+2)}{\rho^2} 
(I-L_{\frac{\rho^2}{2(n+2)}})
[f]_{{\mathscr X}}(X)- 
[{\mathcal A}f ]_{{\mathscr{X}}} (X)   \right|
\leq {K}_1\frac{| \mathbb{X} |}{\rho^2}+{K}_2 \rho, \quad X \in {\mathbb X}.
\label{pointwise generator estimate}
\end{align}
\end{tm}
\begin{proof}
We set $\delta=\frac{\rho^{2}}{2(n+2)}$ and 
fix $X \in {\mathbb X}$. 
We denote the reference point $x(X)\in {\mathscr{X}}$ by $x$ for the simplicity of notation.
Then we have
\begin{align} 
\frac{1}{\delta}  &
(I-L_{\delta})[f]_{{\mathscr X}}(X)- 
[{\mathcal A}f ]_{{\mathscr{X}}} (X) \notag\\
&=
\frac{1}{\delta} \big( f(x)- f(\varphi_{\delta} (x))+\delta bf(x) \big)
\notag \\
&\mbox{ }~~
-V(x)\big( f(x)- f(\varphi_{\delta} (x))+\delta bf(x) \big)
\notag \\
&\mbox{ }~~+\big( \Delta f(x)- \Delta f(\varphi_{\delta} (x))+\delta b\Delta f(x) \big)
\notag \\
&\mbox{ }~~
+\delta \big(V(x)bf(x)-b \Delta f(x) \big)
\notag \\
&\mbox{ }~~-
\frac{1}{\delta} L_{\delta}[f]_{{\mathscr X}}(X)-V(x)f(\varphi_{\delta} (x))
+\Delta f(\varphi_{\delta} (x))
+\frac{1}{\delta}f(\varphi_{\delta}(x)).
\label{t1}
\end{align}
Applying the Taylor expansion formula, we have
\begin{equation}
 f(x) - f(\varphi_{\delta} (x))=-\delta \int_{0}^{1}bf(\varphi_{\delta \theta} (x)){\dd}\theta
 \label{Taylor-824}
 \end{equation}
 and
\begin{equation}
f(x) - f(\varphi_{\delta} (x)) + \delta b f (x) = -\frac{\delta^2}{2} 
\int_{0}^{1}
\big\{ 
{\rm Hess} f (b, b) (\varphi_{\delta \theta} (x )) + \left( \nabla_b  b \right) f ( \varphi_{\delta \theta} (x )) \big \} {\dd}\theta,
\label{Taylor-820}
\end{equation}
where $\nabla$ and ${\rm{Hess}} f$ stand for 
the Levi-Civita connection of the Riemannian metric $g$ and 
the Hessian of $f$, respectively.
In a local coordinate $(x^{(1)}, \ldots, x^{(n)})$ of $M$ with the natural basis 
$\{ {\partial}_{i}={\partial}/{\partial}x^{(i)} \}_{i=1}^{n}$, by using the expansion 
$b=\sum_{i=1}^{n}b^{i}{\partial}_{i}$, 
we may write
$${\rm Hess} f (b, b) 
=
b^{i}b^{j} \frac{\partial^2 f}{\partial x^{(i)} \partial x^{(j)}} -b^kb^i \Gamma_{ki}^j \frac{\partial f}{\partial x^{(j)}},
\quad  
\left( \nabla_b  b \right) f 
=b^i \Big( \frac{\partial b^j}{\partial x^{(i)}} \partial_j + b^j \Gamma_{ji}^k \partial_k \Big) f,
$$
where $\{\Gamma_{ij}^{k}\}_{i,j,k=1}^{n}$ are the Christoffel symbols of the Levi-Civita connection.

It follows from (\ref{Taylor-824}) and (\ref{Taylor-820}) that
\begin{align*}
& \Big \vert \frac{1}{\delta} 
\big( f(x)- f(\varphi_{\delta} (x))+\delta bf(x) \big)
\Big \vert
\leq \frac{\delta}{2}
\big ( \Vert {\rm{Hess}} f \Vert_{\infty}+\Vert (\nabla_{b} b)f \Vert_{\infty} \big)
\end{align*}
and
\begin{align*}
\Big \vert
& -V(x)\big( f(x)- f(\varphi_{\delta} (x))+\delta bf(x) \big)
+\big( \Delta f(x)- \Delta f(\varphi_{\delta} (x))+\delta b\Delta f(x) \big)
\Big \vert
\notag \\
&\leq
2 \delta
\big(
 \Vert V \Vert_{\infty, S(f)} \Vert bf \Vert_{\infty}
+\Vert b\Delta f \Vert_{\infty}
\big).
\end{align*}
Besides, we also have
\begin{equation*}
\big \vert \delta \big( V(x)bf(x)-b \Delta f(x) \big) \big \vert 
\leq \delta \big(  \Vert V \Vert_{\infty, S(f)} \Vert bf \Vert_{\infty}+\Vert b(\Delta f) \Vert_{\infty}
\big).
\end{equation*}

Noting $\max\{(1-\delta V)_, 0 \}=(1-\delta V)$ in $S(f)$,
we expand the final line of
the right-hand side of (\ref{t1}) as 
\begin{align}
-\frac{1}{\delta} 
&
L_{\delta}
[f]_{{\mathscr X}}(X)-V(x)f(\varphi_{\delta} (x))
+\Delta f(\varphi_{\delta} (x))
+\frac{1}{\delta}f(\varphi_{\delta}(x))
\notag \\
&=\frac{\delta V(x)-1}{\delta}
\sum_{Y \in {N}_{\rho, \mathscr{X}}(\delta; X)}
\frac{ \mathfrak{m}(Y)}{ \mathfrak{m} ( \mathcal{N}_{\rho, \mathscr{X}}(\delta; X) )} f(x(Y))
\notag \\
&\mbox{ }~~~~
+\frac{1}{\delta}f(\varphi_{\delta}(x))
-V(x)f(\varphi_{\delta} (x))
+\Delta f(\varphi_{\delta} (x))
\notag \\
&=
\frac{1-\delta V(x)}{\delta \mathfrak{m} ( \mathcal{N}_{\rho,\mathscr{X}}(\delta; X) )} 
\Big \{ \sum_{Y \in  {N}_{\rho, \mathscr{X}}(\delta; X)} 
 \mathfrak{m}(Y)  \big( f(\varphi_{\delta} (x) ) - f(x(Y))  \big) 
\notag \\
&\mbox{ }
\hspace{50mm}
-\int_{B(\varphi_{\delta}(x), \rho)} \left( f(\varphi_{\delta}(x)) - f(z) \right) {\dd} \mathfrak{m}(z) \Big \}
\notag \\
&\mbox{ }~+\big(1-\delta V(x)\big)
\Big \{
\frac{1}{\delta \mathfrak{m} ( \mathcal{N}_{\rho}(\delta; X) )} 
\int_{B(\varphi_{\delta}(x), \rho)} \left( f(\varphi_{\delta}(x)) - f(z) \right) {\dd} \mathfrak{m}(z)
+\Delta f(\varphi_{\delta} (x))
\Big \}
\notag \\
 &\mbox{ }~
+\delta V(x)\Delta f(\varphi_\delta (x)).
\label{expand-2}
\end{align}
Here we easily have
\begin{align*}
\vert  \delta V(x)\Delta f(\varphi_\delta (x))
\vert
\leq \delta 
\Vert V \Vert_{\infty, S(f)} \Vert \Delta f \Vert_{\infty}\leq \Vert \Delta f \Vert_{\infty}.
\end{align*}
To estimate the first term
on the right-hand side of (\ref{expand-2}),
we recall the definition of the Riemannian integral. Indeed, 
taking into account that 
\begin{equation}
B_{\rho-|\mathbb{X} | }(\varphi_{\delta}(x) ) \subset \mathcal{N}_{\rho, \mathscr{X}}(\delta; X) \subset 
B_{\rho + |\mathbb{X}| } (\varphi_{\delta}(x)   ), 
\label{BN}
\end{equation}
we have
\begin{align*}
& 
\sum_{Y  \in N_{\rho, \mathscr{X} }(\delta; X)} 
\mathfrak{m}(Y) \big( f(\varphi_{\delta}(x)) -f(x(Y)) \big) 
\nonumber \\
&
\hspace{10mm}
=
\int_{\mathcal{N}_{\rho, \mathscr{X}}(\delta; X)} \big( f(\varphi_{\delta} (x) )- f(z) \big) {\dd} \mathfrak{m}(z) 
+\sum_{Y \in  N_{\rho, \mathscr{X}}(\delta; X) }  \int_{Y} \left( f(z)-f(x(Y)) \right) {\dd} \mathfrak{m}(z)\\
&
\hspace{10mm}
= \int_{B_{\rho} (\varphi_{\delta}(x))} \big( f(\varphi_{\delta} (x) ) -f(z) \big) {\dd} \mathfrak{m}(z)  
\nonumber \\
&
\hspace{10mm}
\mbox{  }~~~
+\int_{B_{\rho+|\mathbb{X} |} (\varphi_{\delta}(x) ) \backslash B_{\rho} (\varphi_{\delta}(x)) } 
\big( f(\varphi_{\delta}(x)) -f(z) \big) {\dd} \mathfrak{m}(z)  \\
&
\hspace{10mm}
\mbox{ }~~~ 
-\int_{B_{\rho+|\mathbb{X} |}(\varphi_{\delta}(x) ) \backslash \mathcal{N}_{\rho, \mathscr{X} }(\delta; X) }
\big( f(\varphi_{\delta}(x)) -f(z) \big) {\dd} \mathfrak{m}(z) 
\nonumber \\
&
\hspace{10mm}
\mbox{  }~~~
+\sum_{Y \in N_{\rho, \mathscr{X} }(\delta; X) }  \int_{Y} \left( f(z)-f(x(Y)) \right) {\dd} \mathfrak{m}(z). 
\end{align*}
Here we note that there exist positive constants $c_1, c_{2}$ and $c_{3}$
depending only on geometry on $S(f)$ such that for any $0<h<r<1$
\begin{align*}
\mathfrak{m}\big(B_{r+h} ( \varphi_{\delta} (x)) \big) 
- \mathfrak{m}\big( B_{r-h} (\varphi_{\delta}(x) ) \big)
& \leq c_1h r^{n-1} 
\end{align*}
and
\begin{align*}
c_3r^{n} 
\leq \mathfrak{m}\big(
B_r (\varphi_{\delta}(x) ) \big) & \leq c_2r^{n}.
\end{align*}
Combining these estimates with the assumption of $|\mathbb{X} |<\rho/3$, we obtain
\begin{align*}
& \bigg | \frac{1-\delta V(x)}{\delta \mathfrak{m}(\mathcal{N}_{\rho, \mathscr{X}} (\delta; X))} 
\Big \{ \sum_{Y  \in N_{\rho, \mathscr{X}}(\delta; X)} 
\mathfrak{m}(Y) \big( f(\varphi_{\delta}(x) )-f(x(Y)) \big) 
\\
&\mbox{ } \hspace{50mm}
 -\int_{B_{\rho} (\varphi_{\delta}(x))} \big( f(\varphi_{\delta} (x) ) -f(z) \big) {\dd} \mathfrak{m}(z) \Big \} \bigg | 
\\
&\leq \frac{1}{\delta c_3(\rho-|\mathbb{X} |)^n} 
\Big \{ 
\Big | \sum_{Y \in  N_{\rho, \mathscr{X} }(\delta; X) }  
\int_{Y} \big( f(z)-f(x(Y)) \big) {\dd} \mathfrak{m}(z) \Big |
\\
&
\hspace{35mm} 
+
2\ 
\int_{B_{\rho+|\mathbb{X}|} (\varphi_{\delta}(x)) \backslash B_{\rho-|\mathbb{X}|} (\varphi_{\delta}(x) ) } 
\big | f(\varphi_{\delta}(x)) -f(z) \big | {\dd} \mathfrak{m}(z)  
\Big \}
\\
&\leq 
 \frac{2^n}{\delta c_3\rho^n} 
\Big \{
\| f \|_{C^1} |\mathbb{X}| \mathfrak{m}\big(B_{\rho+|\mathbb{X}| }(\varphi_{\delta}(x))
\big)
\\
&\mbox{ }
\hspace{15mm} +
2 \| f \|_{C^1} (\rho+|\mathbb{X} | ) 
\Big( \mathfrak{m}\big(
B_{\rho+|\mathbb{X} |} ( \varphi_{\delta} (x) )
\big)
-
\mathfrak{m}\big(
B_{\rho-|\mathbb{X} | } (\varphi_{\delta}(x))
\big) \Big)
\Big \}
\\
&\leq 
\frac{2^n(4c_1+2^n c_2)(n+2)}{c_3} \| f \|_{C^1} \frac{|\mathbb{X} | }{\rho^2} .
\end{align*}

To estimate the second term on the right-hand side of (\ref{expand-2}), 
we now apply the spherical mean approximation 
of the Laplacian. 
Let us take a Riemannian normal coordinate system
 $(u^{(1)},\ldots, u^{(n)})$ at  $\varphi_{\delta}(x) \in M$ via exponential map 
$\exp_{\varphi_{\delta}(x)}: T_{\varphi_{\delta}(x)} M \rightarrow M$.
More precisely, for an orthonormal basis $\{{\bf{e}}_1, \ldots , {\bf{e}}_n\}$ of $T_{\varphi_{\delta}(x)} M$, 
the map
 \begin{equation*}
 (u^{(1)}, \cdots, u^{(n)}) \stackrel{E}{\longmapsto} \sum_{j=1}^n u^{(j)}{\bf{e}}_j 
 \stackrel{\exp_{\varphi_{\delta}(x)}}{\longmapsto}
 \exp_{\varphi_{\delta}(x)} \left( \sum_{j=1}^n u^{(j)} {\bf{e}}_j \right) 
 \end{equation*}
 gives a diffeomorphism between
 $B^{\mathbb{R}^n}(r)$, the open ball of radius $r$ in $\mathbb R^{n}$
centered at origin and $B_r (\varphi_{\delta}(x))$
 for small $r>0$ and then the pair $(B_r (\varphi_{\delta}(x)), 
 E^{-1}\circ \exp_{\varphi_{\delta}(x)}^{-1} )$ can be regarded 
 as a chart containing $\varphi_{\delta}(x)$.
 
Now take $z \in S(f)$ and $r=\rho <\mathrm{inj}_M(S(f))$.
Applying the Taylor expansion formula
to the function  $\tilde{f}=f\circ\exp_{\varphi_{\delta}(x)}\circ E$ on 
$B^{\mathbb{R}^n}(r) 
\subset \mathbb{R}^n$ at the origin, 
we have 
 for $u=(u^{(1)}, \cdots ,u^{(n)})=E^{-1}\circ   \exp_{\varphi_{\delta}(x)}^{-1} (z) $
 \begin{align*}
 f(z) &= \tilde{f} (u^{(1)}, \cdots , u^{(n)} )
\nonumber \\
&={\tilde f}(0)+\frac{\partial {\tilde f}}{\partial u^{(j)}}(0)u^{(j)}
+\frac{1}{2} \frac{\partial^{2} {\tilde f}}{\partial u^{(j)} \partial u^{(k)}}(0)
u^{(j)}u^{(k)}+{\mathcal T}_{jkl}({\tilde f})(u)
u^{(j)}u^{(k)}u^{(l)},
\end{align*}
where
$$
{\mathcal T}_{jkl}({\tilde f})(u)
=
\int_{0}^{1}{\dd} \theta_{1}
\int_{0}^{\theta_{1}} {\dd} \theta_{2}
\int_{0}^{\theta_{2}} {\dd} \theta_{3}
\frac{\partial^{3} {\tilde f}}{\partial u^{(j)} \partial u^{(k)} \partial u^{(l)} }(\theta_{3} u).
$$

Now we make use of the fact that the Riemannian volume element 
${\sqrt{{\rm det}({\tilde{g}} (u))}}$ in the normal coordinates has the expansion
\begin{align*}
{\sqrt{{\rm det}({\tilde{g}} (u))}}=1-\frac{1}{6} {\rm Ric}_{jk}(0) u^{(j)} u^{(k)}+
{\mathcal G}_{3}\big(B_{\rho} (\varphi_{\delta}(x)); u \big), 
\end{align*}
where $\mathrm{Ric}_{jk}(0):=\mathrm{Ric}(\varphi_{\delta}(x))(e_{j}, e_{k})$, the $(j,k)$-component 
of the Ricci curvature tensor  
at $\varphi_{\delta}(x)$, and 
${\mathcal G}_{3}\big( B_{\rho} (\varphi_{\delta}(x)); u \big)$
is the reminder term. Here we should remark that there exists a
continuous non-negative function $G$
on $M$ such that
${\mathcal G}_{3}\big( B_{\rho} (\varphi_{\delta}(x)); u \big)$
satisfies for all $|u|_{\mathbb R^{n}} < \rho$
\begin{equation}
\sup_{0<\rho'<\rho} \big \vert {\mathcal G}_{3}(B_{\rho'} (\varphi_{\delta'}(x) ); u) \big \vert \leq 
G(\varphi_{\delta}(x)) \vert u \vert_{\mathbb R^{n}}^{3},
\label{volume-reminder-est}
\end{equation} where $\delta'=\frac{(\rho')^{2}}{2(n+2)}$.
See e.g., \cite[Lemma 3.5 in Chapter II]{Sak96} and \cite[Lemma 3.4]{Sak71} for details.
We also mention here that
\begin{equation*}
 \int_{B^{\mathbb{R}^n}(\rho) } u^{(j)} {\dd} u=0, \quad  
\int_{B^{\mathbb{R}^n}(\rho)} u^{(j)} u^{(k)} {\dd}u =\delta_{jk} \frac{\omega_{n}\rho^{n+2}}{n+2},
\end{equation*}
where $\omega_{n}$ is the volume of the unit ball $B^{\mathbb{R}^n}(1)$.

We then obtain 
\begin{align*}
&
\int_{B(\varphi_{\delta}(x), \rho)} 
\left( f(\varphi_{\delta}(x)) - f(z) \right) {\dd} \mathfrak{m}(z) 
\nonumber \\
&=
\int_{B^{\mathbb{R}^n}(\rho)} 
\big( {\tilde f}(0) - {\tilde f}(u) \big) \sqrt{\det ( {\tilde g} (u))} \hspace{0.5mm} {\dd} u 
\notag \\
&=\int_{B^{\mathbb{R}^n}(\rho) }
\Big( - \frac{\partial {\tilde f}}{\partial u^{(j)} } (0) u^{(j)} 
-\frac{1}{2} \frac{\partial^2 {\tilde f}}{\partial u^{(j)} \partial u^{(k)}} (0) u^{(j)} u^{(k)}
+
{\mathcal T}_{jkl}({\tilde f})(u)
u^{(j)}u^{(k)}u^{(l)}
\Big)
\nonumber \\
&\mbox{  } \hspace{18mm} \times
\Big( 1-\frac{1}{6} {\rm Ric}_{jk}(0) u^{(j)} u^{(k)}
+
 {\mathcal G}_{3}(B_{\rho} (\varphi_{\delta}(x)); u)
\Big)
{\dd}u 
\notag\\
&= 
-\frac{\partial {\tilde f}}{\partial u^{(j)}}(0) \int_{B^{\mathbb{R}^n}(\rho) } u^{(j)} {\dd} u 
-\frac{1}{2}\frac{\partial^2 {\tilde f}}{\partial u^{(j)} u^{(k)}} (0) 
 \int_{B^{\mathbb{R}^n}(\rho) } u^{(j)} u^{(k)} {\dd} u 
+
{\mathcal R}_{f}\big(B_{\rho} (\varphi_{\delta}(x))\big)
\nonumber \\
&=-\frac{\omega_{n} \rho^{n+2}}{2(n+2)} \Delta f( \varphi_{\delta}(x))
+{\mathcal R}_{f}\big(B_{\rho} (\varphi_{\delta}(x))\big),
\end{align*}
where
\begin{align*}
{\mathcal R}_{f}\big(B_{\rho} (\varphi_{\delta}(x))) &:=
\frac{1}{6}\frac{\partial {\tilde f}}{\partial u^{(l)}}(0) \mathrm{Ric}_{jk}(0) 
\int_{B^{\mathbb{R}^n}(\rho) }  u^{(j)} u^{(k)} u^{(l)} {\dd} u
\nonumber \\
&~~~~~
+\int_{B^{\mathbb{R}^n}(\rho) }  
{\mathcal T}_{jkl}({\tilde f})(u)
u^{(j)}u^{(k)}u^{(l)} {\dd} u 
\nonumber \\
&~~~~~
-\frac{\partial {\tilde f}}{\partial u^{(j)}}(0)
 \int_{B^{\mathbb{R}^n}(\rho) } 
 u^{(j)} 
 {\mathcal G}_{3}(B_{\rho} (\varphi_{\delta}(x)); u)
 {\dd} u 
\nonumber  \\
&~~~~~
+\frac{1}{12} 
\frac{\partial^2 {\tilde f}}{\partial u^{(j)} u^{(k)}} (0) 
 \mathrm{Ric}_{rs}(0) 
\int_{B^{\mathbb{R}^n}(\rho) } 
u^{(j)} u^{(k)}
u^{(r)} u^{(s)} {\dd} u
\nonumber \\
&~~~~~
-\frac{1}{2} 
\frac{\partial^2 {\tilde f}}{\partial u^{(j)} u^{(k)}} (0) 
\int_{B^{\mathbb{R}^n}(\rho) } 
u^{(j)} u^{(k)} 
 {\mathcal G}_{3}(B_{\rho} (\varphi_{\delta}(x)); u)
{\dd}u
\nonumber \\
&~~~~~
-\frac{1}{6} \mathrm{Ric}_{rs}(0)
\int_{B^{\mathbb{R}^n}(\rho) }  
{\mathcal T}_{jkl}({\tilde f})(u)
u^{(j)}u^{(k)}u^{(l)}
u^{(r)} u^{(s)} {\dd} u 
\nonumber \\
&~~~~~
+\int_{B^{\mathbb{R}^n}(\rho)}  
{\mathcal T}_{jkl}({\tilde f})(u)
u^{(j)}u^{(k)}u^{(l)}
 {\mathcal G}_{3}(B_{\rho} (\varphi_{\delta}(x)); u)
 {\dd}u.
\end{align*}
Repeating the same calculation as above, we also have
\begin{align*}
{\mathfrak m}\big(B_{\rho} (\varphi_{\delta}(x)) \big)=\omega_{n} \rho^{n}-{\rm{Scal}}(0)\frac{\omega_{n}
\rho^{n+2}}{n+2}+\int_{B^{\mathbb R^{n}}(\rho)}
{\mathcal G}_{3}(B_{\rho}(\varphi_{\delta}(x)); u)
 {\dd} u,
\end{align*}
where ${\rm{Scal}}(0)$ stands for the scalar curvature at $\varphi_{\delta}(x)$.
Recalling (\ref{volume-reminder-est}), we observe that
$$\vert {\mathcal R}_{f}\big(B_{\rho} (\varphi_{\delta}(x)) \big)\vert
+\Big \vert
\int_{B^{\mathbb R^{n}}(\rho)}
{\mathcal G}_{3}(B_{\rho} (\varphi_{\delta}(x)); u)
 {\dd} u
\Big \vert
\leq
c_4 \omega_{n} \rho^{n+3},
$$
where the constant $c_4$ depends on $\| f \|_{C^3}$ and 
the geometry on $B_{\rho} (\varphi_{\delta}(x))$.
Using (\ref{BN}), we obtain
\begin{align*}
& \left | (1-\delta V(x)) \Big \{ \frac{1}{\delta \mathfrak{m} ( \mathcal{N}_{\rho, \mathscr{X} }(\delta; X) )} 
\int_{B_{\rho} (\varphi_{\delta}(x))} \left( 
f(\varphi_{\delta}(x)) - f(z) \right) {\dd} \mathfrak{m}(z) 
+\Delta f (\varphi_\delta (x) ) \Big \} \right|
\nonumber \\
&\leq 
\left| \frac{\omega_n \rho^n}{\mathfrak{m}(\mathcal{N}_{\rho, \mathscr{X}} (\delta; X))} -1\right| 
\cdot \vert
\Delta f (\varphi_{\delta}(x))
\vert
+\frac{c_{4}\omega_{n} \rho^{n+3}}{ \delta^2 \mathfrak{m}(\mathcal{N}_{\rho,\mathscr{X}}(\delta; X))} 
\nonumber \\
&\leq 
\frac{  
\vert
\Delta f (\varphi_{\delta}(x))
\vert
}
{\mathfrak{m}(B_{\rho-|\mathbb{X}|} (\varphi_{\delta}(x) )) }
\Big \{
\big | \omega_n \rho^n-\mathfrak{m}(B_{\rho} (\varphi_{\delta}(x) )) \big |
\nonumber \\
&\mbox{ }~~~~~~~~~
+
\mathfrak{m}\big(B_{\rho} (\varphi_{\delta}(x) )\backslash B_{\rho- |\mathbb{X}|} (\varphi_{\delta}(x) ) \big)
+\mathfrak{m}\big( \mathcal{N}_{\rho, \mathscr{X}} (\delta; X) \backslash B_{ \rho-|\mathbb{X}| } (\varphi_{\delta}(x)) \big) 
\Big \}
\nonumber \\
&\mbox{ }~~+ 
\frac{2c_4\omega_{n} (n+2)\rho^{n+1}}{\mathfrak{m}(B_{\rho-|\mathbb{X}|} (\varphi_{\delta}(x)  ))} 
\nonumber \\
&\leq 
\frac{
\big \vert \Delta f (\varphi_{\delta}(x))
\big \vert
}
{c_{3}  (\rho-|\mathbb{X}|)^{n}}
\Big(
\vert
{\rm{Scal}}(0)
\vert
\frac{\omega_n
\rho^{n+2}
}{n+2}
+
c_{4}\omega_{n} \rho^{n+3}
+2c_1 |\mathbb{X}| \rho^{n-1} 
\Big)
+ \frac{2c_4\omega_{n} (n+2)\rho^{n+1}}
{c_3(\rho-|\mathbb{X}|)^{n}}
\nonumber
\\
&\leq 
\frac{2^n}{c_3} \big(  
\vert {\rm{Scal}}(0) \vert
\frac{\omega_n \rho^{2}}{n+2}+c_{4}\omega_{n} \rho^{3}
+2c_1 \frac{|\mathbb{X} | }{\rho} \big) 
\Vert \Delta f \Vert_{\infty}
+\frac{2^{n+1}c_4\omega_{n}(n+2)}{c_3} \rho.
\end{align*}

Putting it all together, we finally obtain our desired estimate
(\ref{pointwise generator estimate}).
\end{proof}
\begin{re}
Combining smoothness of $b$, $V$ and $f$, connectivity of each 
$X\in {\mathbb X}$, and compactness of $S(f)$,
we can also show that the constants $K_1$ and $K_2$ 
in 
{\rm{(\ref{pointwise generator estimate})}}
do not depend on the choice of reference points 
$\mathscr{X}$. 
\end{re}
\begin{co}\label{generator convergence}
Let $\{ \mathbb{X}_{k} \}_{k\in \mathbb N}$ be a sequence of partitions such that $| \mathbb{X}_k | \searrow 0$ 
as $k \rightarrow \infty$.
Then for any subsequence $\{ k(\rho) \}_{\rho>0}$ of $\mathbb N$ such that
$k(\rho) \nearrow \infty$ and $|\mathbb{X}_{k(\rho )} |=o(\rho^2)$ as $\rho \searrow 0$ and 
for any $f \in C_c^\infty (M)$,
\begin{equation}
\lim_{\rho \searrow 0} 
\left\| 
\frac{2(n+2)}{\rho^2} \big(I-L_{\frac{\rho^2}{2(n+2)}
}
\big) [f]_{{\mathscr X}_{k(\rho)}} 
 - [{\mathcal A}f]_{{\mathscr X}_{k(\rho)}} 
\right\|_{C_0({\mathbb G}_{\partial} (\mathbb{X}_{k(\rho)}, \rho))}=0,
\label{gen infty}
\end{equation}
and
\begin{equation}
\lim_{\rho \searrow 0} 
\left\| 
\frac{2(n+2)}{\rho^2} \big(I-L_{\frac{\rho^2}{2(n+2)}
}
\big) {\mathcal P}_{{\mathbb X}_{k(\rho)}} f  -
{\mathcal P}_{{\mathbb X}_{k(\rho)}}
 \left( 
{\mathcal A} f \right)  
\right\|_{L^{p}({\mathbb G}_{\partial} (\mathbb{X}_{k(\rho)}, \rho))}=0.
\label{gen p}
\end{equation}
\end{co}
\begin{proof}
Applying Theorem
\ref{pointwise generator convergence}, 
we easily have (\ref{gen infty}).
Next, we prove (\ref{gen p}).
It is easy to see 
\begin{align*}
\supp 
\left( \frac{2(n+2)}{\rho^2} (I-L_{\frac{\rho^2}{2(n+2)}
}) {\mathcal P}_{{\mathbb X}_{k}}f  
-
{\mathcal P}_{{\mathbb X}_{k}}
\left( {\mathcal A}  f\right)   \right)
\end{align*}
for any $0<\rho \leq 1$ and $|\mathbb{X}_k | \leq 1$.
Then 
\begin{equation*}
\mathfrak{m} \left( 
\supp\left( \frac{2(n+2)}{\rho^2} (I-L_{\frac{\rho^2}{2(n+2)}
}) 
{\mathcal P}_{{\mathbb X}_{k}}f  -
{\mathcal P}_{{\mathbb X}_{k}} \left( {\mathcal A} f \right)   \right) \right)
\end{equation*}
is uniformly bounded by the finite constant ${\mathfrak m}(S(f))$ 
for all  $0<\rho \leq 1$ and $|\mathbb{X}_k |  \leq 1$.
Hence we obtain
\begin{align}
& 
\hspace{-5mm}
\left\| 
\frac{2(n+2)}{\rho^2} (I-L_{\frac{\rho^2}{2(n+2)}
}) 
{\mathcal P}_{{\mathbb X}_{k(\rho)}} f  -
{\mathcal P}_{{\mathbb X}_{k(\rho)}}
\left( {\mathcal A} f\right)  
\right\|_{L^{p}({\mathbb G}_{\partial} (\mathbb{X}_{k(\rho)}, \rho))}
\notag \\
& \leq {\mathfrak m}(S(f))^{1/p}
\left\| 
\frac{2(n+2)}{\rho^2} (I-L_{\frac{\rho^2}{2(n+2)}
}) 
{\mathcal P}_{{\mathbb X}_{k(\rho)}} f  -{\mathcal P}_{{\mathbb X}_{k(\rho)}}
 \left( {\mathcal A} f\right)  
\right\|_{C_0({\mathbb G}_{\partial} (\mathbb{X}_{k(\rho)}, \rho))}
\notag \\
& \leq  {\mathfrak m}(S(f))^{1/p} 
\frac{2(n+2)}{\rho^2} 
\left\| (I-L_{\frac{\rho^2}{2(n+2)}}) \left(
{\mathcal P}_{{\mathbb X}_{k(\rho)}} f  -[f]_{\mathscr{X}_{k(\rho)}} \right)
\right\|_{C_0({\mathbb G}_{\partial} (\mathbb{X}_{k(\rho)}, \rho))} 
\notag
\\
&\quad + {\mathfrak m}(S(f))^{1/p}  \left\| 
\frac{2(n+2)}{\rho^2} \big(I-L_{\frac{\rho^2}{2(n+2)}
}
\big) [f]_{{\mathscr X}_{k(\rho)}} 
 - [{\mathcal A}f]_{{\mathscr X}_{k(\rho)}} 
\right\|_{C_0({\mathbb G}_{\partial} (\mathbb{X}_{k(\rho)}, \rho))} 
\notag \\
&\quad + {\mathfrak m}(S(f))^{1/p}  \left\| [{\mathcal A}f]_{{\mathscr X}_{k(\rho)}} 
-\mathcal{P}_{\mathbb{X}_{k(\rho)}} (\mathcal{A} f )
\right\|_{C_0({\mathbb G}_{\partial} (\mathbb{X}_{k(\rho)}, \rho))}.
\label{p infty}
\end{align}
We observe that 
$\Big \| I-L_{\frac{\rho^2}{2(n+2)}} \Big \|_{
C_0({\mathbb G}_{\partial} (\mathbb{X}_{k(\rho)}, \rho))
\to 
C_0({\mathbb G}_{\partial} (\mathbb{X}_{k(\rho)}, \rho))
} \leq 2$ and for any $X \in \mathbb{X}_{k(\rho)}$
\begin{align*}
\mathcal{P}_{\mathbb{X}_{k(\rho)}} ( f )(X) - [f]_{{\mathscr X}_{k(\rho)}}  (X) 
&= \frac{1}{\mathfrak{m}(X)} \int_{X} ( f(z)- f(x(X))) {\dd} \mathfrak{m}(z) \\
& \leq \sup_{z \in X} | f(z)- f(x(X)) | \\
& \leq \| f \|_{C^1} | \mathbb{X}_{k(\rho)} |.
\end{align*}
Then the first and the third term in (\ref{p infty}) can be estimated by
\begin{align*}
& 
  {\mathfrak m}(S(f))^{1/p} 
\frac{2(n+2)}{\rho^2} 
\left\| (I-L_{\frac{\rho^2}{2(n+2)}}) \left(
{\mathcal P}_{{\mathbb X}_{k(\rho)}} f  -[f]_{\mathscr{X}_{k(\rho)}} \right)
\right\|_{C_0({\mathbb G}_{\partial} (\mathbb{X}_{k(\rho)}, \rho))}  \\
&
\hspace{5mm}
\leq 2{\mathfrak m}(S(f))^{1/p} 
\frac{2(n+2)}{\rho^2} \| f \|_{C^1} | \mathbb{X}_{k(\rho)} |,  
\end{align*}
and
$$
 {\mathfrak m}(S(f))^{1/p}  \left\| [{\mathcal A}f]_{{\mathscr X}_{k(\rho)}} 
-\mathcal{P}_{\mathbb{X}_{k(\rho)}} (\mathcal{A} f )
\right\|_{C_0({\mathbb G}_{\partial} (\mathbb{X}_{k(\rho)}, \rho))} 
 \leq {\mathfrak m}(S(f))^{1/p} \| \mathcal{A}f \|_{C^1} | \mathbb{X}_{k(\rho)} |,
$$
respectively. 
Since $|\mathbb{X}_{k(\rho)}| =o(\rho^2)$, these terms converges to $0$. Consequently, (\ref{gen p}) is deduced from (\ref{gen infty}).
\end{proof}

\subsection{Convergence of semigroups}
\begin{proof}[{\bfseries Proof of Theorems {{\ref{Main-1}}} and {{\ref{Main-3}}}}]
Combining Lemma \ref{approximation-jyunbi}, Corollary \ref{generator convergence}
with the condition {\bf{(A)}} (resp. 
{\bf{(A1)$_{\bm{p}}$}}
and
{\bf{(A2)$_{\bm{p}}$}}), 
we may apply Trotter's approximation theorem (e.g., \cite{Tro58, Kur69}) to have
(\ref{Main conv1}) (resp. (\ref{Main conv2})).

Now we prove Theorem \ref{Main-3}. Noting that
the hypoellipticity of the elliptic operator ${\mathcal A}$ implies
${\rm e}^{-t{\mathcal A}}(C^{\infty}(M)) \subset C^{\infty}(M)$ for all $t\geq 0$,
we may apply Namba \cite{Nam22}.
For given $\rho>0$, we take $k(\rho) \in \mathbb N$ such that $\vert {\mathbb X}_{k(\rho)} \vert <\rho^{2+\alpha}$.
Combining Theorem \ref{pointwise generator convergence} with \cite[Theorem 1]{Nam22},
we have
\begin{align}
&
\Big \Vert
L_{\frac{\rho^2}{2(n+2)}}
^{\lfloor \frac{2(n+2)}{\rho^2}t \rfloor} 
[f]_{{\mathscr X}_{k(\rho)}}
-\big[ {\ee}^{-t\mathcal A }  f \big]_{{\mathscr X}_{k(\rho)}}
\Big \Vert_{C_0({\mathbb G}_{\partial} (\mathbb{X}_{k(\rho)}, \rho))}
\nonumber \\
&
\leq 
\sqrt{\frac{t\rho^{2}}{2(n+2)}} \Big( {K}_{1}(0)
\frac{\vert {\mathbb X}_{k(\rho)}\vert }{\rho^2}+{K}_{2}(0) \rho
+\Vert {\mathcal A}f \Vert_{\infty}
 \Big)
  \nonumber \\
 &\mbox{~~}
+\frac{\rho^{2}}{2(n+2)}
\Big( {K}_{1}(0)\frac{\vert {\mathbb X}_{k(\rho)}\vert }{\rho^2}+{K}_{2}(0) \rho
+\Vert {\mathcal A}f \Vert_{\infty}
 \Big)
+\int_{0}^{t} \Big ( {K}_{1}(s)
\frac{\vert {\mathbb X}_{k(\rho)}\vert }{\rho^2}
+{K}_{2}(s)
\rho \Big ) {\dd}s
\nonumber \\
&\leq {C}' \big \{t^{1/2}\rho+{\rho}^2+t \max_{0\leq s \leq t} \big({K}_{1}(s))
\rho^{\alpha}
+
t \max_{0\leq s \leq t} \big({K}_{2}(s)) \rho 
\big \}
= {C} \rho^{\alpha \wedge 1},
\quad f\in C^{\infty}(M),~t>0,
\nonumber
\end{align}
where ${K}_{i}(s)={K}_{i}
(\Vert {\ee}^{-s{\mathcal A}}f \Vert_{C^{3}}, 
\Vert b \Vert_{\infty}, \Vert \nabla_{b}b \Vert_{\infty}, \Vert V \Vert_{\infty})$, $i=1,2$, $s\geq 0$, are positive
constants appearing in Theorem \ref{pointwise generator convergence}. This gives our desired estimate
(\ref{Main conv3}).
\end{proof}
\begin{re} In the proof of Theorem {\rm{\ref{Main-3}}}, it is a 
key point to find a good
core ${\mathfrak D}$ such that ${\ee}^{-t{\mathcal A}} ({\mathfrak D}) \subset {\mathfrak D}$.
In the case where $M$ is compact, as mentioned above, it suffices to put ${\mathfrak D}=C^{\infty}(M)$.
On the other hand, in the case where $M$ is non-compact, this problem is not so trivial.
We expect that
$$ {\mathfrak D}=\{ f\in C_{0}(M) \cap C^{3}_{b}(M); {\mathcal A}f\in C_{0}(M) \}$$
is a candidate of 
such a core. To check the stability of ${\mathfrak D}$ under the operation 
${\ee}^{-t{\mathcal A}}$, we need to show boundedness of
the third order derivatives of ${\ee}^{-t{\mathcal A}}f$ for $f\in {\mathfrak D}$.
(Note that the first and the second order derivatives of ${\ee}^{-t{\mathcal A}}f$ is obtained by
{\rm{\cite{Tho19, Li21}}} under several conditions on curvature and the derivative of the potential
function $V$.)
On the other hand, it is not so difficult to have  
{\rm{(\ref{pointwise generator estimate})}} for $f\in \mathfrak D$ because the coefficients ${K}_{1}$ and 
${K}_{2}$ depend on the derivatives up to the third order. Hence we conjecture that 
Theorem {\rm{\ref{Main-3}}} still holds in the case where $M$ is non-compact by imposing
additional conditions as mentioned above. 
We will discuss this problem in the future.
\end{re}
\begin{proof}[{\bfseries Proof of Corollary \ref{Main-2}}]
First, we prove (1). For $x \in M$, let $y=x(X_{k(\rho)}(x))\in X_{k(\rho)} \in \mathbb{X}_{k(\rho)} $.
Then we obtain
\begin{align}
& \left| {\ee}^{-t{\mathcal A}}f(x)-
L_{\frac{\rho^2}{2(n+2)}}^{\lfloor \frac{2(n+2)}{\rho^2}t \rfloor}
[f]_{{\mathscr X}_{k(\rho)}}( X_{k(\rho)} (x)) \right|   \notag \\
\leq & \left| {\ee}^{-t{\mathcal A}}f(x)-{\ee}^{-t{\mathcal A}}f(y)
\right| 
+ \left| 
\left[ {\ee}^{-t{\mathcal A}}f \right]_{\mathscr{X}_{k(\rho)}} (X_{k(\rho)}(x)) 
-
L_{\frac{\rho^2}{2(n+2)}}^{\lfloor \frac{2(n+2)}{\rho^2}t \rfloor}
[f]_{{\mathscr X}_{k(\rho)}}( X_{k(\rho)}(x)) 
\right| 
\notag \\
\leq  &
\left| {\ee}^{-t{\mathcal A}}f(x)-{\ee}^{-t{\mathcal A}}f(y)
\right|
+ \left\|
\left[ {\ee}^{-t{\mathcal A}}f \right]_{\mathscr{X}_{k(\rho)}}
-
L_{\frac{\rho^2}{2(n+2)}}^{\lfloor \frac{2(n+2)}{\rho^2}t \rfloor}
[f]_{{\mathscr X}_{k(\rho)}}
\right\|_{C_0(\mathbb{G}_{\partial} (\mathbb{X}_{k(\rho)}, \rho))} .
\label{pointwise1}
\end{align}
Because the function ${\ee}^{-t{\mathcal A}}f$ is continuous and 
$d(x, y) \leq |\mathbb{X}_{k(\rho)} | \rightarrow 0$ as $\rho \rightarrow 0$, 
the first term of the right-hand side in (\ref{pointwise1}) converges to $0$. 
By using (\ref{Main conv1}), the second term in (\ref{pointwise1}) converges to $0$. Thus
we obtain (\ref{pointwise conv}).

Next, we prove (2).
\begin{align}
&\left| \int_{U} {\ee}^{-t{\mathcal A}}f(x) {\dd} \mathfrak{m}(x) -\sum_{X\in {\mathbb X}_{k(\rho)},  X\subset U}
L_{\frac{\rho^2}{2(n+2)}}^{\lfloor \frac{2(n+2)}{\rho^2}t \rfloor}
{\mathcal P}_{{\mathbb X}_{k(\rho)}} f(X) 
{\mathfrak m}(X) \right| \notag \\ 
\leq &
\left| 
\int_{U} {\ee}^{-t{\mathcal A}}f(x) {\dd} \mathfrak{m}(x) - \sum_{X \in \mathbb{X}_{k(\rho)}, X \subset U} 
\int_X  {\ee}^{-t{\mathcal A}}f(x) {\dd} \mathfrak{m}(x) \right| \notag \\
+& \left| \sum_{X \in \mathbb{X}_{k(\rho)}, X \subset U} 
\int_X {\ee}^{-t{\mathcal A}} f(x) {\dd} \mathfrak{m}(x) 
-\sum_{X\in {\mathbb X}_{k(\rho)},  X\subset U}
L_{\frac{\rho^2}{2(n+2)}}^{\lfloor \frac{2(n+2)}{\rho^2}t \rfloor}
{\mathcal P}_{{\mathbb X}_{k(\rho)}} f(X) 
{\mathfrak m}(X)
\right| .
\label{mean est1}
\end{align}
 Because $U$ is a bounded open set and $| \mathbb{X}_{k(\rho)}|  \rightarrow 0$ as $\rho \rightarrow 0$, 
\begin{equation*}
\left| 
\int_{U} {\ee}^{-t{\mathcal A}}f(x) {\dd} \mathfrak{m}(x) - \sum_{X \in \mathbb{X}_{k(\rho)}, X \subset U} 
\int_X  {\ee}^{-t{\mathcal A}}f(x) {\dd} \mathfrak{m}(x) \right| \rightarrow 0  \quad (\rho \rightarrow 0).
\end{equation*}
Then the first term of the right-hand side in (\ref{mean est1}) converges to $0$. 

By using  H\"older's inequality, we obtain
\begin{align*}
& \left| \sum_{X \in \mathbb{X}_{k(\rho)}, X \subset U} 
\int_X  {\ee}^{-t{\mathcal A}}f(x) {\dd} \mathfrak{m}(x)-\sum_{X\in {\mathbb X}_{k(\rho)},  X\subset U}
L_{\frac{\rho^2}{2(n+2)}}^{\lfloor \frac{2(n+2)}{\rho^2}t \rfloor}
{\mathcal P}_{{\mathbb X}_{k(\rho)}} f(X) 
{\mathfrak m}(X) \right| \\
=& \left| \sum_{X \in \mathbb{X}_{k(\rho)}, X \subset U} 
\left( \frac{1}{\mathfrak{m}(X)}  \int_X {\ee}^{-t{\mathcal A}} f(x) {\dd} \mathfrak{m}(x)-
L_{\frac{\rho^2}{2(n+2)}}^{\lfloor \frac{2(n+2)}{\rho^2}t \rfloor}
{\mathcal P}_{{\mathbb X}_{k(\rho)}} f(X)  \right)  {\mathfrak m}(X) \right| \\
\leq & \left( 
 \sum_{X \in \mathbb{X}_{k(\rho)}, X \subset U} 
 \left| 
 \frac{1}{\mathfrak{m}(X)}  \int_X {\ee}^{-t{\mathcal A}} f(x) {\dd} \mathfrak{m}(x)-
L_{\frac{\rho^2}{2(n+2)}}^{\lfloor \frac{2(n+2)}{\rho^2}t \rfloor}
{\mathcal P}_{{\mathbb X}_{k(\rho)}} f(X)   
 \right|^p {\mathfrak m}(X) 
\right)^{1/p} \mathfrak{m}(U)^{1/q} \\
\leq &
\Big \Vert 
{\mathcal P}_{{\mathbb X}_{k(\rho)}} ({\ee}^{-t\mathcal A}  f) -
L_{\frac{\rho^2}{2(n+2)}
}^{\lfloor \frac{2(n+2)}{\rho^2}t \rfloor} {\mathcal P}_{{\mathbb X}_{k(\rho)}} f
\Big \Vert_{L^p({\mathbb G}_{\partial} (\mathbb{X}_{k(\rho)}, \rho))} 
\mathfrak{m}(U)^{1/q}.
\end{align*}
Hence (\ref{Main conv2}) implies that the second term of the right-hand side 
in (\ref{mean est1}) converges to $0$ as $\rho \rightarrow 0$. 
Then the proof of (\ref{mean conv}) is completed. 
\end{proof}
%
\section{Sufficient conditions for 
{\bf{(A)}},  {\bf{(A1)$_{\bm{p}}$}} and {\bf{(A2)$_{\bm{p}}$}}
}
In this section, we 
discuss sufficient conditions for conditions
{\bf{(A)}}, {\bf{(A1)$_{p}$}} and {\bf{(A2)$_{p}$}} on 
the drifted Schr\" odinger operator
${\cal A}_{V}=-\Delta -b +V$ given in the Introduction.
As mentioned in the proof of Theorem 1.3, 
these conditions hold in the case where $M$ is compact.
Hence, in this section, we focus on the case where
$M$ is non-compact.
Although the non-negativity of the
the potential function $V$ was always assumed in previous sections,
we discuss our problems without assuming it in advance in this section.

First of all, we give a basic criterion for condition {\bf{(A)}}
essentially due to Seeley \cite[Theorem 2]{See84}.
We give an outline of the proof for the reader's convenience.

\begin{pr}\label{kihon-pr1}
Let $M$ be a non-compact complete smooth Riemannian manifold, 
and $V$ be a non-negative smooth function on $M$.
Suppose that there exists a sequence of
smooth cut-off functions $\{ \chi_{m} \}_{m=1}^{\infty} \subset C^{\infty}_{c}(M)$
satisfying 
the following three conditions:
\begin{itemize}
\setlength{\leftskip}{0.5cm}
\item[\bf{(a1):}]
$0\leq \chi_{m}(x) \leq 1$ for all $x\in M$ and $m\in {\mathbb N}$.
\item[\bf{(a2):}]
For every compact set $K\subset M$, there exists an integer $m_{0}(K)$ such that 
\\
$\chi_{m}(x)=1$ for all $x\in K$ and $m \geq m_{0}(K)$.
\item[\bf{(a3):}]
There exists a constant $C>0$ such that 
\begin{equation} 
(\Delta+b)\chi_{m}(x) \leq C
\label{keyestimate-825}
\end{equation}
for all $x\in M$ 
and $m\in \mathbb N$.
\end{itemize}
Then condition {\bf{(A)}} holds.
\end{pr}
\begin{proof}
To prove that $(1+\mathcal{A}_{V})(C_c^\infty(M))$ is dense in $C_0 (M)$,
it 
suffices 
to show that $\nu \in C_{0}(M)^{*}$ equals to $0$ provided 
\begin{equation}
\subscripts
 {C_{0}(M)^*}
{\big \langle \nu, (1+{\mathcal A}_{V}) \varphi
 \big \rangle}
{C_{0}(M)}=0 \quad \mbox{for all }\varphi \in C_{c}^{\infty}(M).
\label{nu-weak form}
\end{equation}
By the Riesz-Markov-Kakutani theorem, $\nu \in C_{0}(M)^{*}$ is identified with
a finite signed Borel measure ${\dd} \nu$ on $M$. Moreover combining (\ref{nu-weak form}) with 
the hypoellipticity of elliptic operator ${\mathcal A}_{V}$, we have
${\dd}\nu(x)=u(x) {\dd}{\mathfrak m}(x)$,
where $u\in C^{\infty}(M)\cap L^{1}({\mathfrak m})$ (see e.g., \cite[Proposition 4.5]{IW89}).
Hence we may rewrite 
(\ref{nu-weak form}) as 
\begin{equation}
\int_M u(x) \, (1+ \mathcal{A}_{V})\varphi (x) \, {\dd} {\mathfrak m}(x) =0
\quad \mbox{for all }
\varphi \in C^{\infty}_{c}(M).
\label{dense}
\end{equation}
Since $u\in L^1({\mathfrak m})$, for any  $0<\varepsilon<1$, there exists a compact set $K=K_{\varepsilon} \subset M$ 
such that 
\begin{equation}
\int_{K^c} |u| \, {\dd} {\mathfrak m} \leq \varepsilon \int_{M} |u| \, {\dd}\mathfrak m,
\label{KcM}
\end{equation}
where $K^c$ is the complement of $K$. Namely,
\begin{equation}
\int_{K} |u| \, {\dd}{\mathfrak m} \geq (1-\varepsilon) \int_{M} |u| \, {\dd}\mathfrak m.
\label{KM}
\end{equation}
By {\bf{(a1)}}, {\bf{(a2)}} and (\ref{dense}), we obtain
for all $m \geq m_0(K)$
\begin{align}
\int_{K} |u| \, {\dd}{\mathfrak m} & \leq  
\int_{M} \chi_m(x) |u(x)| \, {\dd}{\mathfrak m}(x)\notag\\
&=  {\color{red}{-}} \int_{M} \chi_m (x) \mathcal{A}^{*}_{V} u(x)  \mathrm{sgn}(u(x)) \, {\dd} {\mathfrak m}(x) \notag\\
&\leq -\int_{M} \chi_m (x) \mathcal{A}^{*}_{0} u(x)  \mathrm{sgn}(u(x)) \, {\dd} {\mathfrak m}(x) \notag\\
&\leq -\int_{M} \chi_{m}(x)\mathcal{A}^{*}_{0} \vert u(x)  \vert \, {\dd} {\mathfrak m}(x) \notag\\
&= - \int_{M} \mathcal{A}_{0}\chi_m (x) |u(x)| \, {\dd} {\mathfrak m}(x)
=  \int_{M} |u(x)| (\Delta +b) \chi_m (x) \, {\dd} {\mathfrak m}(x),
\label{KM-2}
\end{align}
where $$ {\rm{sgn}}(a):=\left\{\begin{array}{lc}
a/\vert a \vert
& \mbox{if }  a\neq 0 \\
0 & \mbox{if } a=0
\end{array}
 \right. 
 ,
 \quad \mathcal{A}^{*}_{V} u=-\Delta u+{\rm{div}}(ub)+Vu$$
and
we used the non-negativity of $V$ for the third line.
We also used 
Kato's inequality 
(cf. \cite[Theorem 2]{See84}) 
$$-\mathcal{A}^{*}_{0}\vert u \vert \geq {\rm{sgn}}(u) (-{\mathcal A}^{*}_{0} u)$$
for the fourth line.

We note here that {\bf{(a2)}} implies
\begin{equation*}
(\Delta +b)\chi_m (x) = 0,  \quad x \in K
\end{equation*}
for all $m\geq m_{0}(K)$. Then by {\bf{(a3)}}, we have
\begin{align}
 \int_{M} |u(x)| (\Delta +b) \chi_m (x) \, {\dd} {\mathfrak m}(x) 
 \leq C \int_{K^{c}}
  |u| 
  \, {\dd} {\mathfrak m}. 
  \label{KM-3}
 \end{align}
Combining (\ref{KcM}), (\ref{KM}), (\ref{KM-2}) with (\ref{KM-3}),
we obtain
\begin{align*}
(1-\varepsilon ) \int_M |u| \, {\dd}{\mathfrak m} & \leq 
 \int_{K} |u| \, {\dd} {\mathfrak m} \\
& \leq   C \int_{K^c} |u| {\dd}{\mathfrak m} 
\leq C\varepsilon \int_M |u| \, {\dd}{\mathfrak m},
\end{align*}
that is, 
\begin{equation*}
0\leq \left \{ C\varepsilon  - (1-\varepsilon)  \right \}  \int_M |u| \, {\dd} {\mathfrak m}.
\end{equation*}
We now take $\varepsilon >0$ small enough.
Then $u$ must satisfy
\begin{equation*}
\int_M |u| \, {\dd}{\mathfrak m}=0,
\end{equation*}
which implies $u=0$. Hence 
we complete the proof.
\end{proof}

In $L^{p}(\mathfrak m)$-setting, the following criterion for conditions
{\bf{(A1)$_{\bm{p}}$}} and {\bf{(A2)$_{\bm{p}}$}} 
is obtained by Shigekawa 
\cite{Shi10, Shi12}. See Theorems 2.1, 4.1 and 4.2 in \cite{Shi10} and Proposition 2 in \cite{Shi12}
for details of the proof.
\begin{pr} \label{pr2}
Suppose that
there exists a constant $\lambda \geq 0$ such that
\begin{equation}
\frac{1}{p}
({\rm{div}}  b)(x) +V(x) \geq -\lambda, \quad 
x\in M,
\label{main p}
\end{equation}
then condition {\bf{(A1)$_{\bm{p}}$}} holds.
Furthermore, suppose that there exist a base point $o\in M$, 
a positive constant $C$ and a positive non-increasing continuous function
$\kappa=\kappa (r): [0,\infty) \to (0,1]$ such that 
\begin{equation}
\int_{0}^{\infty} \kappa (r) \, {\dd}r=\infty \quad \mbox{ and }\quad 
 \kappa(r(x)) br(x) \geq -C,~~ 
x\in M,
\label{Lp cond1}
\end{equation}
then condition {\bf{(A2)$_{\bm{p}}$}} holds,
where $r(x):=d(o,x)$ is the radial function from $o\in M$.
Hence, under {\rm{(\ref{main p})}} and {\rm{(\ref{Lp cond1})}}, 
for any $1<p<\infty$, 
$(-{\mathcal A}_{V},
C^{\infty}_{c}(M))$ is closable in $L^{p}({\mathfrak m})$
and 
its closure
generates a $C_{0}$-semigroup
$\{{\ee}^{-t{\mathcal A}_{V}} \}_{t \geq 0}$ in $L^{p}({\mathfrak m})$ and
${\ee}^{-t{\mathcal A}_{V}}f$ has the Feynman-Kac type functional integral representation
{\rm{(\ref{FK})}} for 
all $f\in C_{c}^{\infty}(M)$. 
Additionally, 
if we impose $V\geq 0$, the semigroup 
$\{{\ee}^{-t{\mathcal A}_{V}} \}_{t \geq 0}$
is Markovian. Namely, for all $f\in L^{p}({\mathfrak m})$, $0\leq f \leq 1$ implies
$0\leq {\ee}^{-t{\mathcal A}_{V}}f \leq 1$.
\end{pr}
\begin{re} If we can take $\lambda=0$ in {\rm{(\ref{main p})}}, the semigroup
$\{{\ee}^{-t{\mathcal A}_{V}} \}_{t \geq 0}$ is contractive in 
$L^{p}({\mathfrak m})$. (Note that it is also shown by {\rm{\cite[Theorem 3]{See84}}}.) 
If we assume $V\geq 0$,
the lower-boundedness of ${\rm div}b$ implies {\rm{(\ref{main p})}}.
However, needless to say, 
$V\geq 0$ and 
{\rm{(\ref{main p})}} does not imply
the lower-boundedness of ${\rm div}b$ 
in general. 
\end{re}
\begin{exm}
The functions 
$$\kappa(r)=1,~ \frac{1}{r}, ~\frac{1}{r\log r}, ~\frac{1}{r \log r \log \log r}, \ldots, \quad r\geq R$$
are typical examples satisfying 
$\int^{\infty}_{R} \kappa(r)\, {\dd}r=\infty$ 
with $R=0, 1, e, e^{e}, \ldots$, 
respectively.  
\end{exm}

For the later purpose, we 
fix a constant $\gamma>1$ and 
take a
smooth function $\phi=\phi_{\gamma}(t): {\mathbb R} \to [0,1]$ 
satisfying $\phi \equiv 1$ on $(-\infty, 1]$,  $\phi \equiv 0$ on $[\gamma, \infty)$ and 
\begin{equation}
-\frac{2}{\gamma-1}\leq \phi'(t) \leq 0, \quad \vert \phi''(t) \vert \leq \frac{8}{(\gamma-1)^{2}} \quad \mbox{for all }
1<t<\gamma.
\label{cutoff-bibun}
\end{equation}
In the following two subsections, we consider the cases where
\begin{itemize}
\item{} $M$ has an empty cut-locus;
\vspace{-2mm}
\item{} $M$ has a variable lower Ricci curvature bound,
\end{itemize}
and give more concrete sufficient conditions for {\bf{(A)}} in terms of the vector field $b$.
\subsection{Manifolds with an empty
cut-locus}
\begin{pr}\label{without-cutlocus}
Let $M$ be a non-compact manifold with an empty cut-locus.
Suppose that there exist some base point $o\in M$ 
and a positive non-increasing smooth function
$\kappa=\kappa (r): (0,\infty) \to (0,1]$ 
such that
\begin{equation}
\sup_{r>0}
\vert \kappa' (r) \vert <\infty, \quad
\int_{0}^{\infty} \kappa (r) 
\, {\dd}r=\infty
\label{kappa-condition}
\end{equation}
and 
the radial function $r(x)=d(o, x)$ satisfies
\begin{equation} \kappa(r(x))\big(\Delta r(x) +br(x)\big) \geq -C 
\int_{0}^{r(x)}
\kappa (r)
\, {\dd}r, 
\quad x\in M \setminus \{ o \}
\label{main infty-827}
\end{equation}
for some positive constant $C$.
Then condition {\bf{(A)}} holds.
\end{pr}
\begin{proof}
We first put $\gamma=4$, for instance, such that the function $\phi$ satisfies
\begin{equation}
-1\leq \phi'(u) \leq 0, \quad \vert \phi''(u)  \vert \leq 1, \quad u \in \mathbb R.
\label{phi-1-estimate}
\end{equation}
We define a sequence of smooth cut-off functions
$\{ \chi_{m} \}_{m=1}^{\infty}$ by
$$\chi_{m}(x):=\phi\Big(\frac{h(r(x))}{m} \Big), \quad x\in M,~m\in \mathbb N,$$
where
$$ h(r)=\int_{0}^{r} \kappa (s) {\dd}s, \quad r \geq 0.$$
We should remark here that (\ref{kappa-condition}) implies $h(r) \nearrow \infty$ as $r\to \infty$.
Since we easily see
$$
\chi_{m}(x)
=
\begin{cases} {\displaystyle{1}} & \text{~if $\displaystyle{r(x)\leq h^{-1}(m)}$},
\vspace{1mm} \\
{\displaystyle{0}} & \text{~if $\displaystyle{r(x)\geq h^{-1}(4m)}$},
\end{cases}
$$
it is sufficient to check (\ref{keyestimate-825})
for all $h^{-1}(m) \leq r(x)\leq  h^{-1}(4m)$. 
Combining a direct calculation with
(\ref{phi-1-estimate}), $0<\kappa \leq 1$, the basic Lipschitz estimate $\vert \nabla r(x) \vert_{T_{x}M} \leq 1$
and condition (\ref{main infty-827}), we obtain
\begin{align*}
\Delta 
&
\chi_{m}(x)+b\chi_{m}(x)
\nonumber \\
 &= \frac{\kappa(r(x))^{2}}{m^{2}}
\phi'' \Big(\frac{h(r(x))}{m}\Big) \vert \nabla {r}(x)
\vert^{2}_{T_{x}M}
+
\frac{\kappa'(r(x))}{m}
\phi' \Big(\frac{h(r(x))}{m}\Big) \vert \nabla {r}(x)
\vert^{2}_{T_{x}M}
\nonumber \\
&\mbox{ }~~+
\frac{\kappa(r(x))}{m}
\phi' \Big(\frac{h(r(x))}{m}\Big) \Delta r(x)
+
\frac{\kappa(r(x))}{m}
\phi' \Big(\frac{h(r(x))}{m}\Big) b r(x)
\nonumber \\
&\leq
\frac{1}{m^{2}}+
\Big( \sup_{r \geq R} \vert \kappa' (r) \vert \Big)-\frac{1}{m}
\big \{ \kappa(r(x)) \cdot \big(\Delta r(x) +br(x)\big) \big \}
\nonumber \\
&\leq C \Big(1+ \frac{h(r(x))}{m} \Big) \leq 5C,
\end{align*}
where we used $h(r(x)) \leq 4m$ for the final inequality.
Thus we have shown {\bf{(a3)}}, and Proposition \ref{kihon-pr1} leads us 
to the desired condition {\bf{(A)}}. This completes the proof.
\end{proof}
As an example of manifolds without a cut-locus, we consider
the case where $M$ has a pole $o\in M$, that is, 
the exponential map $\exp: T_o M \rightarrow M$ is a diffeomorphism.
In this case, we have the following criterion for condition {\bf{(A)}}.
\begin{co} Assume that $M$ has a pole $o\in M$ and there exists a positive constant $C$ such that
\begin{equation}
\mathrm{Ric}(x) \geq -C(1+r(x)^2), \quad x\in M,
\label{Ric-825}
\end{equation}
where $r(x):=d(o,x)$ is the radial function from the pole $o\in M$.
Furthermore assume that there exists a positive constant $C$
such that
 \begin{equation}
 \kappa(r(x)) br(x) \geq -C
 \int_{0}^{r(x)} \kappa (r){\dd}r, \quad x\in M \setminus \{ o \},
 \label{check-827}
 \end{equation}
where the function $\kappa$ is introduced in Proposition
{\rm{\ref{without-cutlocus}}}.
Then condition {\bf{(A)}} holds.
\end{co}
\begin{proof}
By \cite[Appendix]{See84},
the {\it{reverse}} Laplacian comparison theorem
\begin{equation}
\Delta r(x) \geq -C r(x), \quad x\in M\setminus \{o\}
\label{reverse-Laplacian}
\end{equation}
holds for some positive constant $C$. Since the function $\kappa: (0,\infty) \to (0,1]$ 
is non-decreasing and smooth, it follows from (\ref{reverse-Laplacian}) that
$$ \kappa(r(x))\Delta r(x) \geq -C\kappa(r(x)) r(x) \geq -C \int_{0}^{r(x)} \kappa(r) {\dd}r.$$
Combining this estimate with (\ref{check-827}), we have
(\ref{main infty-827}). Thus we may apply
Proposition \ref{without-cutlocus} to complete the proof.
\end{proof}
\begin{exm} If $\kappa \equiv 1$, {\rm{(\ref{check-827})}} is equivalent to 
$$br(x) \geq -Cr(x), \quad x\in M\setminus \{ o\}.$$
If there exists a sufficiently large $R>0$ such that
$\kappa(r)=\frac{1}{r}$ and $\kappa(r)=\frac{1}{r\log r}$ for all $r>R$, 
we may read {\rm{(\ref{check-827})}} as 
$$ br(x) \geq -Cr(x) \big(1+\log r(x) \big) \quad \mbox{and }~
br(x) \geq -Cr(x) \log r(x) \big(1+\log \log r(x) \big), 
$$
respectively.
\end{exm}
\subsection{Manifolds with variable lower Ricci curvature bounds}
We consider the case where the Ricci curvature is bounded from below
by a (possibly unbounded) nonpositive function of the radial function
$r(x):=d(o,x)$ from some base point $o\in M$. 
Throughout this subsection, we assume (\ref{Ric-825}).
Thanks to \cite[Theorem 2.1]{BS18}, there exist a smooth exhaustion 
function ${\mathfrak r}: M \to [0, \infty)$ and constants $0<D_{1}<D_{2}$ 
and $D_{3}>0$ such that
\begin{align}
& D_{1} r(x)^{2}\leq {\mathfrak r}(x) \leq D_{2} \max\{1, r(x)^{2}\}, \quad x\in M,
\label{r-tau}
\\
& \vert \nabla {\mathfrak r}(x) \vert 
\leq D_{3} r(x), \quad x\in M \setminus {\overline{B_{1}(o)}},
\label{r-tau2}
\\
&\vert \Delta {\mathfrak r}(x) \vert
\leq D_{3} r(x)^{2}, \quad x\in M \setminus {\overline{B_{1}(o)}}.
\label{r-tau3}
\end{align}
\begin{co}\label{pr1}
Assume {\rm{(\ref{Ric-825})}} for some base point $o\in M$.
Furthermore assume that there exist
two constants $C>0$ 
and $R\geq 1$ 
such that
\begin{equation} b{\mathfrak r}(x)
\geq -Cr(x)^{2} , 
\quad x\in M \setminus {\overline{B_{R}(o)}}.
\label{main infty}
\end{equation}
Then condition {\bf{(A)}} holds. In particular, if 
the vector field $b$ satisfies
\begin{equation}
\vert b(x) \vert_{T_{x}M} \leq C \big (1+r(x) \big), \quad x\in M
\label{linear-drift}
\end{equation}
for some $C>0$, 
{\rm{(\ref{main infty})}} and hence
condition {\bf{(A)}} hold.
\end{co}
\begin{proof}
We put $\gamma=D_{2}/D_{1}$ and define a sequence of smooth cut-off functions
$\{ \chi_{m} \}_{m=1}^{\infty}$ by
$$\chi_{m}(x):=\phi\Big(\frac{{\mathfrak r}(x)}{D_{1}m^{2}} \Big), \quad x\in M,~m\in \mathbb N.$$
Recalling (\ref{r-tau}) and noting $\gamma>1$, we easily see
$$
\chi_{m}(x)
=
\begin{cases} {\displaystyle{1}} & \text{~if $\displaystyle{r(x)\leq \gamma^{-1/2}m}$},
\vspace{1mm} \\
{\displaystyle{0}} & \text{~if $\displaystyle{r(x)\geq \gamma^{1/2}m}$},
\end{cases}
$$
and thus $\Delta \chi_{m}(x)+b\chi_{m}(x)=0$ holds for $r(x)< \gamma^{-1/2}m$ and 
$r(x)>\gamma^{1/2}m$.

Hence it is sufficient to check condition {\bf{(a3)}} for  
$\gamma^{-1/2}m \leq r(x)\leq  \gamma^{1/2}m$.
Combining a direct calculation with (\ref{cutoff-bibun}), (\ref{r-tau2}), (\ref{r-tau3})
and (\ref{main infty}), we have
\begin{align}
\Delta \chi_{m}(x)+b\chi_{m}(x) &= \frac{1}{D_{1}^{2}m^{4}}
\phi'' \Big(\frac{{\mathfrak r}(x)}{D_{1}m^{2}} \Big) \vert \nabla {\mathfrak r}(x)
\vert^{2}_{T_{x}M}
+\frac{1}{D_{1}m^{2}} \phi' \Big(\frac{{\mathfrak r}(x)}{D_{1}m^{2}} \Big)\Delta {\mathfrak r}(x)
\nonumber \\
&\mbox{ }~~+
\frac{1}{D_{1}m^{2}}
\phi' \Big(\frac{{\mathfrak r}(x)}{D_{1}m^{2}} \Big) b{\mathfrak r}(x)
\nonumber \\
&\leq \frac{8D_{3}^{2}}{(D_{2}-D_{1})^{2}m^{4}}r(x)^{2}
+\frac{2D_{3}}{(D_{2}-D_{1})m^{2}}r(x)^{2}
\nonumber \\
&\mbox{ }~~+
\frac{2C}{(D_{2}-D_{1})m^{2}}r(x)^{2}.
\label{check-825}
\end{align}
Since $r(x)^{2} \leq  \gamma m^{2}$, the right-hand side (\ref{check-825})
is bounded from above by a positive constant $C$ independent of $m$.
It means that we have shown {\bf{(a3)}}, and thus Proposition \ref{kihon-pr1} leads us 
to the desired condition {\bf{(A)}}.

Besides, it follows from 
(\ref{r-tau2}) and (\ref{linear-drift}) that
\begin{align*}
b{\mathfrak r}(x) &\geq -\vert b(x) \vert_{T_{x}M} \vert \nabla {\mathfrak r}(x) \vert_{T_{x}M}
\nonumber \\
& \geq -C(1+r(x)) D_{3} r(x)
\nonumber \\
& \geq -CD_{3}(r(x)+r(x)) r(x)=-Cr(x)^{2}, \quad x\in M \setminus {\overline{B_{1}(o)}}.
\end{align*}
Thus we have shown (\ref{main infty}) with $R=1$. This completes the proof.
\end{proof}
%
\section{Examples}
In this section we study examples of a sequence of partitions and 
drift vector fields to satisfy conditions 
{\bf{(A)}}, 
{\bf{(A1)$_{\bm{p}}$}},
{\bf{(A2)$_{\bm{p}}$}},
{\bf{(B)}} and {\bf{(C)}} on two typical manifolds: Euclidean spaces and model manifolds. 
\subsection{Euclidean spaces}
We consider the case $M=\mathbb{R}^n$. 
Let $x=(x^{(1)}, \ldots ,x^{(n)})$ be the standard Euclidean coordinates 
and we write 
$$
r(x)=\Big \{ \sum_{i=1}^{n} (x^{(i)})^{2} \Big \}^{1/2}, \quad
b(x)=\sum_{i=1}^{n}b^i(x) \frac{\partial}{\partial x_i}.$$
We take a sequence of partitions $\mathbb{X}_k=\{ X_{\mathbf{i}}^k \}_{{\mathbf{i}} \in \mathbb{Z}^d}$, $k\in \mathbb N$, 
by
\begin{equation*}
X_{(i_1, \ldots , i_n)}^k= \left[ \frac{i_1}{k}, \frac{i_1+1}{k} \right) \times \cdots\times  \left[ \frac{i_n}{k} , \frac{i_n+1}{k} \right),
\quad {\mathbf{i}}=(i_{1}, \ldots, i_{n}).
\end{equation*}
Since $| \mathbb{X}_k |=\frac{\sqrt{n}}{k}$, it is easy to check that each
$ \mathbb{X}_k $ satisies conditions {\bf{(B)}} and {\bf{(C)}}.

We first find a sufficient condition on the vector filed $b$
for conditions 
{\bf{(A1)$_{\bm{p}}$}}
and
{\bf{(A2)$_{\bm{p}}$}}
by applying 
Proposition \ref{pr2} with $\kappa \equiv 1$.
In this case, conditions (\ref{main p}) and 
(\ref{Lp cond1})
are rewritten by 
\begin{equation}
\sum_{i=1}^n \frac{\partial b^i}{\partial x^{(i)}}(x)  \geq -\gamma, \quad x\in \mathbb R^{n}
\label{b-108-2}
\end{equation}
and 
\begin{equation}
\sum_{i=1}^n b^i(x) x^{(i)} \geq -C r(x), \quad x\in \mathbb R^{n}\setminus \{0 \},
\label{b-108-1}
\end{equation}
respectively. For example, we consider a vector field $b$ given by 
\begin{equation}
b^i(x)=c^i_0(x) +c^i_1(x) x^{(i)}+c^i_2(x) (x^{(i)})^3 +\cdots +c^i_{k_i}(x)(x^{(i)})^{2k_i-1},\quad k_{i}\in \mathbb N,~
i=1,\ldots, n,
\label{b-108}
\end{equation}
where $c^i_j=c^{i}_{j}(x)$, $i=1,\ldots,n$, $j=1,\ldots, k_{i}$, are smooth functions on $\mathbb R^{n}$.
Note that we did not use the Einstein summation convention in (\ref{b-108}).
Here we further assume that for each $i=1,\ldots, n$, the functions $c^{i}_{j}(x)$, $j=1,\ldots, k_{i}$, 
are independent of $x^{(i)}$. Then (\ref{b-108-2}) and (\ref{b-108-1}) hold provided that
$$c^i_1(x), \ldots , c^i_{k_i}(x) \geq 0, \quad i=1,\ldots, n,~x\in \mathbb R^{n}$$
and 
$$\big \vert  (c^1_0(x), \ldots, c^n_0(x)) \big \vert_{\mathbb R^{n}} \leq C^\prime, 
\quad 
x\in \mathbb R^{n}$$
with some positive constant $C^\prime$, respectively.

Next, we find condition {\bf{(A)}} by applying Proposition \ref{without-cutlocus}.
In this case,  the condition (\ref{main infty-827}) with $\kappa \equiv1$ is rewritten by 
\begin{equation}
\sum_{i=1}^n b^i(x) x^{(i)} \geq  -Cr(x)^{2} -(n-1), \quad x\in {\mathbb R}^{n}.
\label{R^n(A)}
\end{equation}
For example, we also consider the vector field $b$ given by (\ref{b-108}).
By a direct calculation, we easily see that (\ref{b-108}) satisfies
(\ref{R^n(A)}) provided that
$$\big \vert (c_0^1(x), \ldots, c_0^n(x))\big \vert_{\mathbb R^{n}} \leq C^{\prime}r(x),
\quad x\in \mathbb R^{n}$$
and 
$$
c^i_1(x) \geq -C^{\prime},
\quad
c^i_2(x), \ldots c^i_{k_i}(x) \geq 0,
\quad i=1,\ldots, n,~~x\in \mathbb R^{n}$$ with some 
positive constant $C^\prime$.

Since $\mathbb{R}^n$ is regarded as a model manifold, other examples of the vector field $b$ satisfying 
{\bf{(A)}}, 
{\bf{(A1)$_{\bm{p}}$}}
and
{\bf{(A2)$_{\bm{p}}$}}
can be found in Section \ref{model (A)} below.

\subsection{Model manifolds}
\subsubsection{Basic facts for model manifolds}
For $r_0 \in (0, +\infty]$ and a smooth positive function $\psi =\psi(r)$ on $(0, r_0)$, 
let $(M, g_{\psi})$ be an $n$-dimensional {\it{model manifold}} with weight $\psi$, that is, 
there exists one chart on $M$ that covers all of $M$ and the image of this chart in $\mathbb{R}^n$ is 
$B^{\mathbb{R}^n}(r_0)=\{ x \in \mathbb{R}^n ~;~ |x|_{\mathbb{R}^n} <r_0 \}$. 
The metric $g_{\psi} $ in the polar coordinates $(r, \theta)=(r, \theta^1 , \ldots , \theta^{n-1})$ in the above chart 
has the form
\begin{equation*}
g_{\psi} (r,\theta)={\dd}r^2 +\psi(r)^2g_{\mathbb{S}^{n-1}}(\theta),
\end{equation*}
where $g_{\mathbb{S}^{n-1}}$ is the standard Riemannian metric 
on the $(n-1)$-dimensional unit sphere $\mathbb{S}^{n-1}$ which has the form 
\begin{equation*}
g_{\mathbb{S}^{n-1}}(\theta) =\gamma_{ij}(\theta) {\dd}\theta^i {\dd}\theta^j.
\end{equation*}
To avoid the singularity at the origin $o$, we always assume 
$$\lim_{r \searrow 0} \psi(r)=0, \quad
\lim_{r \searrow 0} \psi^\prime(r)=1, \quad \lim_{r \searrow 0} \psi^{\prime \prime}(r)=0.$$
Moreover, when $r_0<\infty$, we assume 
$$\lim_{r \nearrow r_0} \psi(r)=0, \quad
\lim_{r \nearrow r_0} \psi^\prime(r)=-1, \quad \lim_{r \nearrow r_0} \psi^{\prime \prime}(r)=0$$ 
to ensure the completeness and non-singularity
 of the manifold.
 Some of typical manifolds can be regarded as a model manifold. Indeed, 
 \begin{itemize}
\item{} ${\mathbb R}^{n}$ with $r_0=\infty$ and $\psi(r)=r$;
\vspace{-2mm}
\item{} ${\mathbb S}^{n}$ (without a pole) with $r_0=\pi$
and $\psi(r)=\sin r$, $0< r < \pi$;
\vspace{-2mm}
\item{} $\mathbb{H}^n$ with $r_0=\infty$ and $\psi(r)=\sinh r$ 
\end{itemize}
are model manifolds.
We refer to \cite[Section 3.10]{Gri09} for the precise definition and basic results of model manifolds.

On a model manifold $M$, the Riemannian volume measure $\mathfrak{m}$ is 
given in the polar coordinates by 
\begin{equation*}
{\dd} \mathfrak{m}=\psi(r)^{n-1} {\dd}r {\dd} \theta,
\end{equation*}
where ${\dd}\theta$ stands for the Riemannian volume measure on $\mathbb{S}^{n-1}$. 
The Laplacian has the form
\begin{equation*}
\Delta =\frac{\partial^2}{\partial r^2} + 
(n-1)\frac{\psi'(r)}{\psi(r)}
 \frac{\partial}{\partial r} 
+\frac{1}{\psi(r)^2} \Delta_{\mathbb{S}^{n-1}}.
\end{equation*}
In particular, 
\begin{equation*}
\Delta r=(n-1) \frac{\psi^\prime(r)}{\psi (r)}.
\end{equation*}

\subsubsection{Ricci curvature on model manifolds}
When $n$=2, we can compute the Ricci curvature on a model manifold by using the weight function $\psi$ explicitly.
Indeed, the Christoffel symbol $\Gamma_{ij}^k$ can be given by
\begin{equation*}
\Gamma_{ij}^r
= \left\{ \begin{array}{cl}
-\psi(r) \psi^{\prime} (r) & \mbox{  if } (i,j)=(\theta , \theta) \vspace{2mm}\\
0   & \mbox{  otherwise}
\end{array}
\right. , ~
\Gamma_{ij}^{\theta}
= \left\{ \begin{array}{cl}
 \frac{\psi ^\prime (r)}{\psi (r) }& \mbox{  if } (i,j)=(r , \theta)  \mbox{ or } (\theta, r) \vspace{2mm} \\
0   & \mbox{  otherwise}
\end{array}
\right. .
\end{equation*}
Since the component of Ricci curvature tensor $R_{ij}$ is given by
\begin{equation*}
R_{ij}= \partial_k \Gamma_{ij}^k -\partial_j \Gamma_{ik}^k 
+\Gamma_{ij}^k \Gamma_{km}^m -\Gamma_{im}^k \Gamma_{jk}^m,
\end{equation*}
we obtain
\begin{equation*}
(R_{ij})=\left( 
\begin{array}{cc}
-\frac{\psi^{\prime \prime} (r)}{\psi (r)}  & 0 \\
0 &  -\psi (r) \psi^{\prime \prime} (r) 
\end{array}
\right).
\end{equation*}
Hence, we obtain
\begin{equation}
\mathrm{Ric} \geq -
\frac{\psi^{\prime \prime}(r)}{\psi(r)}.
\label{Ric lower}
\end{equation}

When $\psi$ is polynomial, that is, for some positive constant $C_1$ and for some $\alpha \in \mathbb{R}$
\begin{equation*}
\psi(r)=C_1 r^{\alpha}
\end{equation*}
for all large $r>0$, (\ref{Ric lower}) implies that
\begin{equation*}
\mathrm{Ric} \geq  -\alpha (\alpha-1)r^{-2}
\end{equation*}
for all large $r>0$. 

When $\psi$ is exponential, that is, for some positive constant $C_1$ and some $\alpha, \beta \in \mathbb{R}$, 
\begin{equation*}
\psi(r) =C_1 \exp( \alpha r^{\beta})
\end{equation*}
for all large $r>0$, (\ref{Ric lower}) implies that 
\begin{equation*}
\mathrm{Ric} \geq -
 \alpha \beta r^{\beta -2} \big( (\beta -1) +\alpha \beta r^{\beta} \big).
\end{equation*}
In particular,  if $\alpha >0$ and $\beta>2$, then
\begin{equation*}
\mathrm{Ric} \geq -C^\prime r^{2\beta-2}
\end{equation*}
for some $C^\prime >0$ and  all large $r>0$, which fails the variable lower Ricci bound (\ref{Ric-825}).

\subsubsection{Partition of model manifolds}
In this section, we give an example of the partition of the model manifold $(M, g_{\psi})$ satisfying 
conditions {\bf{(B)}} and 
{\bf{(C)}}. 
First of all, let us construct a partition of $\mathbb{S}^{n-1}$ by induction in dimension. 
We first consider the case $n=2$, that is, $\mathbb{S}^1\simeq [0, 2\pi)$. For any $K\in \mathbb{N}$, we take
a partition consisting of $K$-pieces given by
\begin{equation}
\left[ 0, \frac{2\pi}{K} \right), ~ \left[ \frac{2\pi}{K} , 2\frac{2\pi}{K} \right),  \cdots , \left[ (K-1) \frac{2\pi}{K}, 2\pi \right).
\label{S^1}
\end{equation}
Next, as a partition of $\mathbb{S}^2$, take a \textit{spherical suspension} (cf. \cite[Section 3.6.3]{BBI01})
of each piece in (\ref{S^1})
as a subset of $\mathbb{S}^2$ and decompose it with same interval $\frac{\pi}{K}$ 
in the extended angle. This partition is nothing but a partition of $\mathbb{S}^2$ by $K$ 
longitude lines
and $K$ latitude lines. 
Noting that for any piece of this partition $X \subset \mathbb{S}^2$, 
\begin{equation*}
\mathrm{diam}(X) \leq \frac{3 \pi}{K}, \quad \mathrm{vol}_{\mathbb{S}^2}(X) \leq \frac{2\pi^2}{K^2}.
\end{equation*}
we repeat this procedure to construct a partition of $\mathbb{S}^{n-1}$. 
Then the given partition $\mathbb{Y}_K$
consisting of $K^{n-1}$ elements satisfies for any $Y\in \mathbb{Y}_K$
\begin{equation*}
\mathrm{diam}(Y) \leq \frac{n\pi}{K}, \quad \mathrm{vol}_{\mathbb{S}^{n-1}}(Y) \leq 2 \left( \frac{\pi}{K} \right)^{n-1}.
\end{equation*}

Now we construct a partition of a model manifold $M=(M, g_{\psi})$ as follows:
Let $\Pi:
M\backslash \{ o \} \rightarrow \mathbb{S}^{n-1}$ be the canonical projection, that is, 
\begin{equation*}
\Pi(r, \theta) =\theta, \quad (r,\theta)\in M\backslash \{ o \}.
\end{equation*}
For $l, m \in  \mathbb{N}$, we set 
\begin{equation*}
A_l (m) = B_{\frac{m}{l}}(o) 
\backslash 
B_{\frac{m-1}{l}}(o), \quad
\psi_l(m)=\max_{\frac{m-1}{l} \leq r \leq \frac{m}{l}}\psi(r). 
\end{equation*}
(Note that $A_l(m)=\emptyset$ if $r_0<\infty$ and $\frac{m-1}{l} \geq r_0$.)

Next we define a partition $\{ X_{k}^{l,m} \}_{k=1,\ldots, K^{n-1} }$ of $A_l(m)$ by
\begin{equation*}
X_{k}^{l, m}=\Pi^{-1}
(Y_k) \cap A_l(m), \quad k=1, \ldots, K^{n-1}.
\end{equation*}
We  note that 
\begin{equation*}
\mathrm{diam}(X^{l,m}_k) \leq \frac{1}{l} +\mathrm{diam}(Y_k)\psi_l(m) 
\leq \frac{1}{l} + \frac{n\pi}{K}\psi_l(m)
\end{equation*}
and 
\begin{equation*}
\mathfrak{m}( X_{k,l,m} ) \leq \frac{1}{l} \mathrm{vol}_{\mathbb{S}^{n-1}}(Y_k) \psi_l(m)^{n-1} \leq 
 \frac{2}{l} \left( \frac{\pi}{K} \right)^{n-1}  \psi_l(m)^{n-1}.
\end{equation*}
Now choose $K=K(l,m):=\lceil l\psi_l(m) n \pi \rceil$. Then for all $k=1, \ldots, K(l,m)^{n-1}$ and 
$l, m \in \mathbb{N}$, we obtain 
\begin{equation*}
\mathrm{diam}(X^{l,m}_{k}) \leq \frac{2}{l}, \quad \mathfrak{m}(X^{l,m}_k) \leq \frac{2}{n^{n-1} l^n}.
\end{equation*}
Consequently,  for each $l\in \mathbb{N}$,  $\mathbb{X}_l:= \{ X_{k}^{l,m} \}_{m\in \mathbb{N}, k=1, \ldots,  K(l,m)}$ is the partition of $M$.
It is easy to check conditions {\bf{(B)}} and {\bf{(C)}}.

\subsubsection{A sufficient condition for {\bf{(A)}}, 
{\bf{(A1)$_{\bm{p}}$}} and
{\bf{(A2)$_{\bm{p}}$}}
}
\label{model (A)}
In this section, we study a sufficient condition on a vector field $b$ on the non-compact model manifold 
$(M, g_{\psi})$ for conditions 
(\ref{main p}), (\ref{Lp cond1}) and  (\ref{main infty-827}) with $\kappa \equiv 1$.
For a smooth vector field $b=b(r, \theta)$ on $M$, let $b^r, b^\theta_1, \ldots, b^\theta_{n-1}$ 
be the coefficients of $b$ in the polar coordinates $(r, \theta^{(1)}, \ldots ,\theta^{(n-1)})$. Namely, they satisfy 
\begin{equation*}
b(r,\theta) =b^r(r,\theta)\frac{\partial}{\partial r}+b^\theta_1(r, \theta )\frac{\partial}{\partial \theta^{(1)}}+ \cdots +
b^\theta_{n-1}(r, \theta) \frac{\partial}{\partial \theta^{(n-1)}}.
\end{equation*}
Then conditions
(\ref{main p}) and 
(\ref{Lp cond1}) 
are rewritten by
\begin{align*}
& (n-1) \frac{\psi^\prime (r)}{\psi(r)} b^r(r,\theta) +\frac{\partial b^r}
{\partial r} 
(r, \theta)
+
\mathrm{div}_{\mathbb{S}^{n-1}}( b^\theta(r, \theta)) \geq  -\gamma
\end{align*}
and $b^r(r,\theta) \geq -C$, respectively, where 
\begin{equation*}
\mathrm{div}_{\mathbb{S}^{n-1}}( b^\theta (r, \theta) ) =\frac{1}{\sqrt{ \det (\gamma)}} \sum_{i=1}^{n-1} 
\frac{\partial}{\partial  \theta^{(i)}} \left( 
\sqrt{ \det (\gamma) } b^\theta_i (r, \theta) \right).
\end{equation*}
Also, condition (\ref{main infty-827}) with $\kappa \equiv 1$ is rewritten by
\begin{equation*}
b^r(r,\theta )+(n-1)\frac{\psi^\prime(r)}{\psi(r)} \geq -Cr.
\end{equation*}

We now study a sufficient condition of $b$ in the cases where 
$\psi$ are polynomial and exponential separately.
\vspace{2mm} \\
{\bf{(I) Polynomial case:}}~For a positive constant $C_1$ and some $\alpha \in \mathbb{R}$, suppose
\begin{equation*}
\psi(r)=C_1 r^{\alpha}
\end{equation*}
for all large $r$. 
In this case, 
for the vector filed $b$ satisfying (\ref{main p}), assume that
\begin{equation}
\mathrm{div}_{\mathbb{S}^{n-1}}( b^\theta (r, \theta)) \geq -C^\prime
\label{lower div}
\end{equation}
for some constant $C^\prime$. For example, this condition is true if $b^{\theta}_1 , \ldots  , b^\theta_{n-1}$ 
are constants for all large $r>0$. Then the condition (\ref{main p}) is rewritten by
\begin{equation}
(n-1) \frac{\alpha}{r} b^r +\frac{\partial b^r}{\partial r}
\geq C^\prime-\gamma=:-\gamma^\prime.
\label{model poly (A)'}
\end{equation}
For example, we consider the case 
\begin{equation}
b^r(r,\theta)=c_0(\theta) +c_1(\theta) r+\cdots + c_k(\theta) r^k,
\label{br-11-24}
\end{equation}
where 
$c_0, \ldots, c_k$ are smooth functions on $\mathbb{S}^{n-1}$.
Then conditions (\ref{lower div}) and  
$$
\left( (n-1)\alpha +l \right)  c_{l}(\theta) \geq 0, \quad \theta \in \mathbb{S}^{n-1}, \quad
l=2,\ldots, k
$$ 
imply (\ref{model poly (A)'}) with some $\gamma^\prime$. Hence condition 
{\bf{(A1)$_{\bm{p}}$}}
holds.
Moreover, if $c_k$ in (\ref{br-11-24}) is positive, 
then $b$ satisfies also (\ref{Lp cond1}) with some constant $C$, which implies condition 
{\bf{(A2)$_{\bm{p}}$}}.

On the other hand, condition (\ref{main infty-827}) is rewritten by
\begin{equation*}
b^r(r, \theta) \geq -Cr -(n-1) \frac{\alpha}{r}.
\end{equation*}
For example, 
if 
\begin{equation}
b^r(r, \theta) \geq -C^\prime r
\label{poly (A)}
\end{equation}
for some $C^\prime >0$,
then $b$ satisfies condition (\ref{main infty-827}), whence the condition {\bf{(A)}} holds.
\vspace{2mm} \\
{\bf{(II) Exponential case:}}
For a positive constant $C_1$ and  some $\alpha, \beta \in \mathbb{R}$, suppose
\begin{equation*}
\psi(r)=C_1\exp(\alpha r^{\beta} )
\end{equation*}
for all large $r$. 
To find a vector field $b$ satisfying 
{\bf{(A1)$_{\bm{p}}$}}
and
{\bf{(A2)$_{\bm{p}}$}}, we assume (\ref{lower div}).
Then condition (\ref{main p}) is rewritten by
\begin{equation}
(n-1) \alpha \beta r^{\beta -1} b^r +\frac{\partial b^r}{\partial r}
\geq -\gamma^\prime.
\label{model exp (A)'}
\end{equation}
For example, a vector filed $b$ with (\ref{lower div}) and $b^r(r, \theta)=c(\theta) r$ with a smooth function $c(\theta)$ on $\mathbb{S}^{n-1}$ 
satisfying 
\begin{equation*}
c(\theta)\alpha \beta \geq 0, \quad \theta \in \mathbb{S}^{n-1}
\end{equation*}
satisfies (\ref{model exp (A)'}), and thus condition 
{\bf{(A1)$_{\bm{p}}$}}
holds.
Moreover, if $c(\theta)$ is positive, then $b$ satisfies also (\ref{Lp cond1}), which implies condition 
{\bf{(A2)$_{\bm{p}}$}}.
 
The condition (\ref{main infty-827}) can be rewritten by
\begin{equation}
b^r(r, \theta)  \geq -Cr -(n-1) \alpha \beta r^{\beta-1} .
\label{exponential b1}
\end{equation}
If $\beta \leq 2$, the leading term in the right-hand side of (\ref{exponential b1}) is $-Cr$. Hence, for example, 
 if
\begin{equation}
b^r(r, \theta) \geq -C^\prime r,
\label{exp (A)1}
\end{equation}
then $b$ satisfies the condition (\ref{main infty-827})
for some positive constant $C^\prime$ and then the condition {\bf{(A)}} holds.

If $\beta >2$, then the leading term in the right-hand side of (\ref{exponential b1}) is 
$-(n-1)\alpha \beta r^{\beta -1}$. Hence, for example, $b$ satisfies the condition (\ref{main infty-827}) 
if 
\begin{equation}
b^r (r, \theta) \geq -(n-1) \alpha \beta r^{\beta-1}.
\label{exp (A)2}
\end{equation}
 In particular, the case where $n=2$,  $\alpha>0$ and $\beta >2$, the vector filed $b$ 
with (\ref{lower div}) and  (\ref{exp (A)2}) gives a new example of 
$\mathcal{A}$ with {\bf{(A)}} without the variable lower Ricci bound (\ref{Ric-825}).
%
%
%
%
\vspace{2mm} \\
\noindent
{\bf Acknowledgements.}
The authors are grateful to Professor Atsushi Kasue for useful discussions on the Laplacian 
comparison theorem. They also thank Professors Atsushi Atsuji and Jun Masamune for 
giving valuable comments. 
The first author was partially supported by JSPS Grant-in-Aid for Scientific
Research (C) No. 17K05215, (C) No. 22K03280 and (S) No. 22H04942.
The second author was partially supported by JSPS Grant-in-Aid for Scientific
Research (C) No. 17K05300, (C) No. 20K03639 and (B) No. 21H00988.

\end{document}